\documentclass[reqno]{amsart} 
 
\usepackage{hyperref}
\usepackage{amsmath}
\usepackage{amssymb}
\usepackage{amsfonts}
\usepackage{amsthm}
\usepackage{latexsym}
\usepackage{esint}
\usepackage{mathtools}
\usepackage{mathrsfs}
\usepackage{graphicx}
\usepackage{wrapfig}
\usepackage{bbm}
\usepackage{color}
\usepackage{bm}
\usepackage[top=1in, bottom=1.25in, left=1.10in, right=1.10in]{geometry}
\usepackage{todonotes}
\usepackage{stackengine}
\allowdisplaybreaks

\newcommand\ocirc[1]{\ensurestackMath{\stackon[1pt]{#1}{\mkern2mu\circ}}}

\newtheorem{theorem}{Theorem}[section]

\newtheorem{proposition}[theorem]{Proposition}

\theoremstyle{definition}
\newtheorem{definition}[theorem]{Definition}

\theoremstyle{remark}
\newtheorem{remark}[theorem]{Remark}

\numberwithin{equation}{section}

\newcommand{\bx}{\mathbf{x}}
\newcommand{\by}{\mathbf{y}}

\newcommand{\bq}{\mathbf{q}}
\newcommand{\bu}{\mathbf{u}}
\newcommand{\bfu}{\mathbf{u}}

\newcommand{\bff}{\mathbf{f}}

\newcommand{\bv}{\mathbf{v}}
\newcommand{\bn}{\mathbf{n}}

\newcommand{\dy}{\, \mathrm{d}\by}
\newcommand{\dd}{\,\mathrm{d}}
\newcommand{\dq}{\, \mathrm{d} \mathbf{q}}
\newcommand{\dx}{\, \mathrm{d} \mathbf{x}}

\newcommand{\dt}{\, \mathrm{d}t}

\newcommand{\divx}{\mathrm{div}_{\mathbf{x}}}
\newcommand{\divq}{\mathrm{div}_{\mathbf{q}}}

\newcommand{\delx}{\Delta_{\mathbf{x}}}
\newcommand{\dely}{\Delta_{\mathbf{y}}}

\newcommand{\nabx}{\nabla_{\mathbf{x}}}
\newcommand{\naby}{\nabla_{\mathbf{y}}}
\newcommand{\nabq}{\nabla_{\mathbf{q}}}

\newcommand{\Delx}{\Delta_{\mathbf{x}}}
\newcommand{\Dely}{\Delta_{\mathbf{y}}}

\newcommand{\bT}{\mathbb{T}}
\newcommand{\R}{\mathbb{R}}

\newcommand{\Oeta}{\Omega_{\eta}}
\newcommand{\Ozeta}{\Omega_{\zeta}}

\begin{document}

%\title[Serrin condition for Polymeric fluid-structure interaction]{Serrin condition for the macroscopic polymeric fluid-structure interaction system}

\title[Serrin condition for Polymeric fluid-structure interaction]{Conditionally strong solution for macroscopic polymeric {SSS} interaction}

%%    Information for first author
%\author{Dominic Breit}
%%    Address of record for the research reported here
%\address{Department of Mathematics, Heriot-Watt University, Edinburgh, EH14 4AS, United Kingdom}
%\email{d.breit@hw.ac.uk}
%
%\thanks{The authors would like to thank S. Schwarzacher and E. S\"uli for valuable suggestions.}
\author{Prince Romeo Mensah}
\address{Institut f\"ur Mathematik,
Technische Universit\"at Clausthal,
Erzstra{\ss}e 1,
38678 Clausthal-Zellerfeld}
\email{prince.romeo.mensah@tu-clausthal.de \\ orcid ID: 0000-0003-4086-2708}
%\thanks{The author would like to thank Dominic Breit for many useful discussions.}

%    General info
\subjclass[2020]{76D03; 74F10; 35Q30; 35Q84; 82D60}

\date{\today}


\keywords{Ladyzhenskaya-Prodi-Serrin  condition, Oldroyd-B, Fluid-Structure interaction, Polymeric fluid}

\begin{abstract} 
The system under study is a solute-solvent-structure (SSS) interaction problem for the interaction of a dilute three-dimensional Oldroyd-B polymeric fluid with a two-dimensional viscoelastic shell. We show that a unique global strong solution to this system exists under the condition that the classical Ladyzhenskaya--Prodi--Serrin criterion holds for the velocity field and that the shell displacement is essentially bounded in time with values in the space of continuously differentiable functions. No requirement is needed for the polymer number density and the extra stress tensor for the solute component. 
\end{abstract}

\maketitle
%\tableofcontents

\section{Introduction}

We consider the Oldroyd-B \cite{oldroyd1950formulation} model for the flow of a dilute polymeric fluid interacting with a flexible structure in  the closure of the deformed spacetime cylinder
\begin{align*}
I\times\Oeta:=\bigcup_{t\in I}\{t\}\times\Oeta \in \mathbb{R}^{1+3}.
\end{align*}
Here, $I:=(0,T)$ is a fixed time horizon and $\Oeta:={\Omega_{\eta(t)}}$ is a time-dependent flexible spatial domain whose closure is obtained through the parametrization of the boundary $\omega\subset \mathbb{R}^2$ of some reference configuration $\Omega\subset \mathbb{R}^3$. Further details will be given in Section \ref{sec:prelims}.
 Our goal is to find a structure displacement function $\eta:(t, \by)\in I \times \omega \mapsto   \eta(t,\by)\in \mathbb{R}$, a fluid velocity field $\mathbf{u}:(t, \mathbf{x})\in I \times \Oeta \mapsto  \mathbf{u}(t, \mathbf{x}) \in \mathbb{R}^3$, a pressure function $p:(t, \mathbf{x})\in I \times \Oeta \mapsto  p(t, \mathbf{x}) \in \mathbb{R}$, a polymer number density $\rho :(t, \mathbf{x} )\in I \times \Oeta  \mapsto \rho(t, \mathbf{x} ) \in \mathbb{R}$
and an extra stress tensor $\bT :(t, \mathbf{x} )\in I \times \Oeta  \mapsto \bT (t, \mathbf{x} ) \in \mathbb{R}^{3\times3}$
 such that the system of equations 
\begin{align}
\divx \bu=0, \label{divfree}
\\
\partial_t \rho+ (\mathbf{u}\cdot \nabx) \rho
= 
\Delx \rho 
,\label{rhoEqu}
\\
\partial_t \bu  + (\mathbf{u}\cdot \nabx)\mathbf{u} 
= 
\delx \bu -\nabx p
+\bff
+
\divx   \bT, \label{momEq}
\\
\partial_t^2 \eta - \partial_t\dely \eta + \dely^2 \eta = g - ( \mathbb{S}\bn_\eta )\circ \bm{\varphi}_\eta\cdot\bn \,\det(\naby\bm{\varphi}_\eta), 
\label{shellEQ}
\\
\partial_t \bT + (\mathbf{u}\cdot \nabx) \bT
=
(\nabx \bu)\bT + \bT(\nabx \bu)^\top - 2(\bT - \rho \mathbb{I})+\Delx \bT \label{solute}
\end{align}
holds on $I\times\Oeta\subset \mathbb R^{1+3}$ (with \eqref{shellEQ} defined on $I\times\omega\subset \mathbb R^{1+2}$). Here,
\begin{align*}
\mathbb{S}=  \nabx \bu +(\nabx \bu)^\top  -p\mathbb{I}+ \bT,
\end{align*}
is a tensor field, the vector $\bn_\eta$ is the normal at $\partial\Oeta$, $\mathbb{I}$ is the identity matrix and where all underlining dimensional parameters have been set to one for simplicity.
We complement \eqref{divfree}--\eqref{solute} with the following initial and boundary conditions
\begin{align}
&\eta(0,\cdot)=\eta_0(\cdot), \qquad\partial_t\eta(0,\cdot)=\eta_\star(\cdot) & \text{in }\omega,
\\
&\bu(0,\cdot)=\bu_0(\cdot) & \text{in }\Omega_{\eta_0},
\\
&\rho(0,\cdot)=\rho_0(\cdot),\quad\bT(0,\cdot)=\bT_0(\cdot) &\text{in }\Omega_{\eta_0},
\label{initialCondSolv}
\\
& 
\bn_{\eta}\cdot\nabx\rho=0,\qquad
\bn_{\eta}\cdot\nabx\bT=0 &\text{on }I\times\partial\Omega_{\eta}.
\label{bddCondSolv}
\end{align}
Furthermore, we impose periodicity on the boundary of $\omega$ and the following condition
\begin{align} 
\label{interface}
&\bu\circ\bm{\varphi}_\eta=(\partial_t\eta)\bn & \text{on }I\times \omega
\end{align}
at the interface between the polymeric fluid and the structure with normal vector $\bn$.
\\
The two unknowns $\rho$ and $\bT$ for the solute component of the polymer fluid are related via the identities
\begin{align} 
\label{macroClosure}
\bT(t, \bx)= \int_{B} f(t,\bx,\bq)\bq\otimes\bq \dq,
\qquad
\rho(t, \bx)= \int_{B} f(t,\bx,\bq) \dq
\end{align}
where for  $B=\mathbb{R}^3$ with elements $\bq\in B$, the function $f$ is the probability density function ($f\geq0$ a.e. on $\overline{I}\times\Oeta\times B$) satisfying the
Fokker--Planck equation
\begin{align}
\label{fokkerPlanck}
\partial_t f+\divx(\bu f)+\divq((\nabx\bu)\bq f)=\Delx f+\divq(M\nabq(f/M))
\end{align}
in $I\times\Oeta\times B$ for a Hookean dumbbell spring potential and Maxwellian
\begin{align*} 
U\Big(\frac{1}{2}|\bq|^2 \Big)=\frac{1}{2}|\bq|^2, \qquad\qquad M=\frac{\exp(-U(\tfrac{1}{2}|\bq|^2))}{\int_B\exp(-U(\tfrac{1}{2}|\bq|^2))\dq},
\end{align*}
respectively.

Very recently, the analysis of polymeric fluid-structure interaction problems was started in \cite{breit2021incompressible}. Whereas Newtonian fluid-structure interaction problems involve a solvent and a structure, a polymeric fluid-structure interaction problem involves the mutual interactions between a solute, a solvent, and a structure. Analyzing such a system, therefore, requires techniques from two hitherto separate fields of analysis in continuum mechanics, i.e. the analysis of polymeric fluids and fluid-structure interactions.  

In \cite{breit2021incompressible}, the authors constructed distributional solutions for the coupling between (1) the $3$D Fokker--Planck equation \eqref{fokkerPlanck} giving the mesoscopic description of the probability distribution of the solute or polymer, (2) the  $3$D incompressible Navier--Stokes equation \eqref{divfree} and \eqref{momEq} giving the macroscopic description of the solvent evolution, and (3) a  $2$D structure modeled by a shell equation \eqref{shellEQ} of Koiter type (where the term $ \Dely^2 \eta- \partial_t\Dely \eta$ is replaced by the gradient of the so-called \textit{Koiter energy}). The uniqueness of these distributional solutions is unknown but not expected. However, the solutions exist until potential degeneracies occur with the Koiter energy or with the structure deformation.
When the  $2$D Koiter shell in \cite{breit2021incompressible} is replaced by the $2$D viscoelastic shell equation \eqref{shellEQ}, the extension to the existence of a unique local-in-time strong solution was then shown in \cite{breit2023existence}. 
Note that for fixed spatial domains subject to periodic boundary conditions, one can construct solutions that are spatially more regular \cite{breit2021local}  than strong solutions. The corresponding result for the system with a structure displacement remains an interesting open problem even in lower dimensions.

The macroscopic closure of the Fokker--Planck equation \eqref{fokkerPlanck} gives rise to \eqref{rhoEqu} and \eqref{solute} by using the identities \eqref{macroClosure}. For the fully macroscopic solute-solvent-structure Oldroyd-B system \eqref{divfree}-\eqref{solute}, distributional solutions and solutions that are satisfied pointwise a.e. in spacetime have been constructed in \cite{mensah2023weak} when the solute-solvent subsystem are posed in $2$D and the structure subsystem is posed in $1$D. Here, provided no degeneracies occur with the structure deformation, both classes of solutions exist globally in time. In \cite{mensah2024vanishing}, the same $2$D/$1$D setting but for corotational fluids where $\nabx\bu$ in \eqref{solute} is replaced by its anti-symmetric part $\tfrac{1}{2}(\nabx\bu-(\nabx\bu)^\top)$ is considered. It is shown that any family of strong solutions parametrized
by the center-of-mass diffusion coefficient converges to a weak solution of the same system without center-of-mass diffusion (i.e. $\Delx\rho=0$ and $\Delx\bT=0$) but with essentially bounded polymer number density and extra stress. 

\section{Preliminaries and main results}
\label{sec:prelims}
For any two non-negative quantities $F$ and $G$, we write $F \lesssim G$  if there is a constant $c>0$  such that $F \leq c\,G$. If $F \lesssim G$ and $G\lesssim F$ both hold, we use the notation $F\sim G$.  The scaler matrix product of the matrices $\mathbb{A}=(a_{ij})_{i,j=1}^d$ and $\mathbb{B}=(b_{ij})_{i,j=1}^d$ is denoted by $\mathbb{A}:\mathbb{B}=\sum_{ij}a_{ij}b_{ji}$.
The symbol $\vert \cdot \vert$ may be used in four different contexts. For a scalar function $f\in \mathbb{R}$, $\vert f\vert$ denotes the absolute value of $f$. For a vector $\bff\in \mathbb{R}^d$, $\vert \bff \vert$ denotes the Euclidean norm of $\bff$. For a square matrix $\mathbb{F}\in \mathbb{R}^{d\times d}$, $\vert \mathbb{F} \vert$ shall denote the Frobenius norm $\sqrt{\mathrm{trace}(\mathbb{F}^T\mathbb{F})}$. Also, if $S\subseteq  \mathbb{R}^d$ is  a (sub)set, then $\vert S \vert$ is the $d$-dimensional Lebesgue measure of $S$.

For $I:=(0,T)$, $T>0$, and $\eta\in C(\overline{I}\times\omega)$ satisfying $\|\eta\|_{L^\infty(I\times\omega)}\leq L$ where $L>0$ is a constant, we define for $1\leq p,r\leq\infty$,
\begin{align*} 
L^p(I;L^r(\Omega_\eta))&:=\Big\{v\in L^1(I\times\Omega_\eta):\substack{v(t,\cdot)\in L^r(\Omega_{\eta(t)})\,\,\text{for a.e. }t,\\\|v(t,\cdot)\|_{L^r(\Omega_{\eta(t)})}\in L^p(I)}\Big\},\\
L^p(I;W^{1,r}(\Omega_\eta))&:=\big\{v\in L^p(I;L^r(\Omega_\eta)):\,\,\nabx v\in L^p(I;L^r(\Omega_\eta))\big\}.
\end{align*} 
Higher-order Sobolev spaces can be defined accordingly. For $k>0$ with $k\notin\mathbb N$, we define the fractional Sobolev space $L^p(I;W^{k,r}(\Oeta))$ as the class of $L^p(I;L^r(\Omega_\eta))$-functions $v$ for which 
\begin{align*}
\|v\|_{L^p(I;W^{k,r}(\Oeta))}^p
&=\int_I\bigg(\int_{\Oeta} \vert v\vert^r\dx
+\int_{\Oeta}\int_{\Oeta}\frac{|v(\bx)-v(\bx')|^r}{|\bx-\bx'|^{d+k r}}\dx\dx'\bigg)^{\frac{p}{r}}\dt
\end{align*}
is finite. Accordingly, we can also introduce fractional differentiability in time for the spaces on moving domains.

\subsection{Setup}
The reference spatial domain  $\Omega \subset \mathbb{R}^3$ has a boundary $\partial\Omega\subset \mathbb{R}^{2}$ that may consist of a flexible part $\omega\subset\mathbb{R}^{2}$ and a rigid part $\Gamma\subset \mathbb{R}^{2}$. However, because the analysis at the rigid part is significantly simpler, we shall identify the whole of $\partial \Omega$ with $\omega$ and endow it with periodic boundary conditions. Let $I:=(0,T)$ represent a time interval for a given constant $T>0$. The time-dependent  displacement of the structure is given by $\eta:\overline{I}\times\omega\rightarrow(-L,L)$ where $L>0$ is a fixed length of the tubular neighbourhood of $\partial\Omega$ given by
\begin{align*}
S_L:=\{\bx\in \mathbb{R}^3\,:\, \mathrm{dist}(\bx,\partial\Omega
)<L \}.
\end{align*}
For some $k\in\mathbb{N}$ large enough, we now assume that $\partial\Omega$  is parametrized by an injective mapping $\bm{\varphi}\in C^k(\omega;\mathbb{R}^3)$ with $\naby \bm{\varphi}\neq0$ such that
\begin{align*}
\partial{\Omega_{\eta(t)}}=\big\{\bm{\varphi}_{\eta(t)}:=\bm{\varphi}(\by)+\bn(\by)\eta(t,\by)\, :\, t\in I, \by\in \omega\big\}.
\end{align*}
The set $\partial{\Omega_{\eta(t)}}$ represents the boundary of the flexible domain at any instant of time $t\in I$ and the vector $\bn(\by)$ is a unit normal at the point $\by\in \omega$. 
We also let $\bn_{\eta(t)}(\by)$ be the corresponding normal of $\partial{\Omega_{\eta(t)}}$ at the spacetime point $\by\in \omega$ and $t\in I$. Then for $L>0$ sufficiently small, we note that $\bn_{\eta(t)}(\by)$ is close to $\bn(\by)$ and $\bm{\varphi}_{\eta(t)}$ is close to $\bm{\varphi}$. As a result,  it follows that
%\begin{align*}
%\partial_y \bm{\varphi}_{\eta(t)} \neq0 \quad\text{ if }& d=2,\quad 
%\partial_{y_1} \bm{\varphi}_{\eta(t)}\times
%\partial_{y_2} \bm{\varphi}_{\eta(t)} \neq0 \quad\text{ if } d=3,
%\\&\bn(\by)\cdot \bn_{\eta(t)}(\by)\neq 0,  
%\end{align*}
\begin{align*}
\partial_{y_1} \bm{\varphi}_{\eta(t)}\times
\partial_{y_2} \bm{\varphi}_{\eta(t)} \neq0 \quad\text{ and } \quad
\bn(\by)\cdot \bn_{\eta(t)}(\by)\neq 0
\end{align*}
for $\by\in \omega$ and $t\in I$. Thus, in particular, there is no loss of strict positivity of the Jacobian determinant provided that $\Vert \eta\Vert_{L^\infty(I\times\omega)}<L$.

For the interior points, we  transform the  reference domain $\Omega$ into a time-dependent moving domain $\Omega_{\eta(t)}$  whose state at time $t\in\overline{I}$ is given by
\begin{align*}
\Omega_{\eta(t)}
 =\big\{
 \bm{\Psi}_{\eta(t)}(\bx):\, \bx \in \Omega 
  \big\}.
\end{align*}
Here,
\begin{align*}
\bm{\Psi}_{\eta(t)}(\bx)=
\begin{cases}
\bx+\bn(\by(\bx))\eta(t,\by(\bx))\phi(s(\bx))     & \quad \text{if } \mathrm{dist}(\bx,\partial\Omega)<L,\\
    \bx & \quad \text{elsewhere } 
  \end{cases}
\end{align*}
is the Hanzawa transform with inverse $\bm{\Psi}_{-\eta(t)}$ and where for a point $\bx$ in the neighbourhood of $\partial\Omega$, the vector $\bn(\by(\bx))$ is the unit normal at the point $\by(\bx)=\mathrm{arg min}_{\by\in\omega}\vert\bx -\bm{\varphi}(\by)\vert$. Also, $s(\bx)=(\bx-\bm{\varphi}(\by(\bx)))\cdot\bn(\by(\bx))$ and $\phi\in C^\infty(\mathbb{R})$ is a cut-off function that is $\phi\equiv0$ in the neighbourhood of $-L$ and $\phi\equiv1$ in the neighbourhood of $0$. Note that $\bm{\Psi}_{\eta(t)}(\bx)$ can be rewritten as
\begin{align*}
\bm{\Psi}_{\eta(t)}(\bx)=
\begin{cases}
\bm{\varphi}(y(\bx))+\bn(\by(\bx))[s(\bx)+\eta(t,\by(\bx))\phi(s(\bx)) ]    & \quad \text{if } \mathrm{dist}(\bx,\partial\Omega)<L,\\
    \bx & \quad \text{elsewhere. } 
  \end{cases}
\end{align*}
The transform $\bm{\Psi}_\eta$ and its inverse 
 $\bm{\Psi}_\eta=\bm{\Psi}_{-\eta}$  satisfy the following properties, see \cite{breit2022regularity, BMSS} for details. If for some $\ell,R>0$, we assume that
\begin{align*}
\Vert\eta\Vert_{L^\infty(\omega)}
+
\Vert\zeta\Vert_{L^\infty(\omega)}
< \ell <L \qquad\text{and}\qquad
\Vert\naby\eta\Vert_{L^\infty(\omega)}
+
\Vert\naby\zeta\Vert_{L^\infty(\omega)}
<R
\end{align*}
holds, then for any  $s>0$, $\varrho,p\in[1,\infty]$ and for any $\eta,\zeta \in B^{s}_{\varrho,p}(\omega)\cap W^{1,\infty}(\omega)$ (where $B^{s}_{\varrho,p}$ is a Besov space), we have that the estimates
\begin{align}
\label{210and212}
&\Vert \bm{\Psi}_\eta \Vert_{B^s_{\varrho,p}(\Omega\cup S_\ell)}
+
\Vert \bm{\Psi}_\eta^{-1} \Vert_{B^s_{\varrho,p}(\Omega\cup S_\ell)}
 \lesssim
1+ \Vert \eta \Vert_{B^s_{\varrho,p}(\omega)},
\\
\label{211and213}
&\Vert \bm{\Psi}_\eta - \bm{\Psi}_\zeta  \Vert_{B^s_{\varrho,p}(\Omega\cup S_\ell)} 
+
\Vert \bm{\Psi}_\eta^{-1} - \bm{\Psi}_\zeta^{-1}  \Vert_{B^s_{\varrho,p}(\Omega\cup S_\ell)} 
\lesssim
 \Vert \eta - \zeta \Vert_{B^s_{\varrho,p}(\omega)}
\end{align}
and
\begin{align}
\label{218}
&\Vert \partial_t\bm{\Psi}_\eta \Vert_{B^s_{\varrho,p}(\Omega\cup S_\ell)}
\lesssim
 \Vert \partial_t\eta \Vert_{B^{s}_{ \varrho,p}(\omega)},
\qquad
\eta
\in W^{1,1}(I;B^{s}_{\varrho,p}(\omega))
\end{align}
holds uniformly in time with the hidden constants depending only on the reference geometry, on $L-\ell$ and $R$. 
\subsection{Concepts of solutions and main results}
We make clear in this section, the various notions of a solution we discuss in this paper and state the main result. Let us begin with a weak solution for which the system \eqref{divfree}-\eqref{solute} of equations are each satisfied weakly in the sense of distributions.
\begin{definition}[Weak solution]
\label{def:weaksolmartFP}
Let $(\bff, g, \rho_0, \bT_0, \bu_0, \eta_0, \eta_\star)$
be a dataset that satisfies
\begin{equation}
\begin{aligned}
\label{mainDataForAll}
&\bff \in L^2(I;L^{2}_\mathrm{loc}(\mathbb{R}^3)),
\qquad g\in L^2(I;L^{2}(\omega)) ,
\\&
\eta_0 \in W^{2,2}(\omega) \text{ with } \Vert \eta_0 \Vert_{L^\infty( \omega)} < L, \quad \eta_\star \in L^{2}(\omega),
\\&\bu_0 \in L^{2}_{\divx}(\Omega_{\eta_0} )\text{ is such that }\bu_0 \circ \bm{\varphi}_{\eta_0} =\eta_\star \bn \text{ on } \omega,
\\&
\rho_0\in L^{2}(\Omega_{\eta_0}), \quad
\bT_0\in L^{2}(\Omega_{\eta_0}),
\\&
\rho_0\geq 0,\,\, \bT_0>0 \quad \text{a.e. in } \Omega_{\eta_0}.
\end{aligned}
\end{equation}
We call
$(\eta, \bu, \rho,\bT)$  
a \textit{weak solution} of   \eqref{divfree}--\eqref{interface} with data $(\bff, g, \rho_0, \bT_0, \bu_0, \eta_0, \eta_\star)$  if: 
\begin{itemize}
\item[(a)] the following properties 
\begin{align*}
&\eta\in  W^{1,\infty}\big(I;L^{2}(\omega)  \big)  \cap W^{1,2}\big(I;W^{1,2}(\omega)  \big) \cap L^{\infty}\big(I;W^{2,2}(\omega)  \big),
\\&
\Vert\eta\Vert_{L^\infty(I\times\omega)}<L,
\\
&\bu\in
L^{\infty} \big(I; L^{2}(\Oeta) \big)\cap L^2\big(I;W^{1,2}_{\divx}(\Oeta)  \big),
\\
&
\rho \in   L^{\infty}\big(I;L^{2}(\Oeta)  \big)
\cap 
L^2\big(I;W^{1,2}(\Oeta)  \big),
\\
&
\bT \in   L^{\infty}\big(I;L^{2}(\Oeta)  \big)
\cap 
L^2\big(I;W^{1,2}(\Oeta)  \big),
\\&
\rho\geq 0,\,\, \bT>0 \quad \text{a.e. in } I\times\Oeta
\end{align*}
holds;
\item[(b)] for all  $  \psi  \in C^\infty (\overline{I}\times \R^3  )$, we have
\begin{equation}
\begin{aligned} 
\label{weakRhoEq}
\int_I  \frac{\mathrm{d}}{\dt}
\int_{\Oeta } \rho\psi \dx \dt 
&=
\int_I
\int_{\Oeta}[\rho\partial_t\psi + (\rho\bu\cdot\nabx) \psi] \dx\dt
\\&-
\int_I\int_{\Oeta}\nabx \rho \cdot\nabx \psi  \dx\dt;
\end{aligned}
\end{equation}
\item[(c)] for all  $  \mathbb{Y}  \in C^\infty (\overline{I}\times \R^3  )$, we have  
\begin{equation}
\begin{aligned} 
\label{weakFokkerPlanckEq}
\int_I  \frac{\mathrm{d}}{\dt}
\int_{\Oeta } \bT:\mathbb{Y} \dx \dt 
&=
\int_I
\int_{\Omega_{\eta }}[\bT :\partial_t\mathbb{Y} + \bT:(\bu\cdot\nabx) \mathbb{Y}] \dx\dt
\\&+
\int_I\int_{\Oeta}
[(\nabx \bu )\bT  + \bT (\nabx \bu )^\top]:\mathbb{Y} \dx\dt
\\&
-2\int_I\int_{\Oeta }(\bT :\mathbb{Y}  - \rho  \mathrm{tr}(\mathbb{Y} ))\dx\dt
\\&-
\int_I\int_{\Oeta}\nabx \bT ::\nabx \mathbb{Y}  \dx\dt
\end{aligned}
\end{equation}
where $\nabx\bT::\nabx\mathbb{Y}=\sum_{i=1}^3\partial_{x_i}\bT:\partial_{x_i}\mathbb{Y}$;
\item[(d)] for all  $(\bm{\phi},\phi)  \in C^\infty_{\divx} (\overline{I}\times \R^3  )\otimes  C^\infty (\overline{I}\times \omega )$ with $\bm{\phi}(T,\cdot)=0$, $\phi(T,\cdot)=0$ and $\bm{\phi}\circ \bm{\varphi}_\eta= \phi \bn$, we have
\begin{equation}
\begin{aligned}
\label{weakFluidStrut}
\int_I  \frac{\mathrm{d}}{\dt}\bigg(
\int_{\Oeta } \bu\cdot\bm{\phi} \dx 
+
\int_\omega\partial_t\eta\phi\dy\bigg) \dt 
&=
\int_I
\int_{\Oeta }[\bu \cdot\partial_t\bm{\phi} + \bu\cdot(\bu\cdot\nabx) \bm{\phi}] \dx\dt
\\&-
\int_I\int_{\Oeta }\big[\nabx \bu  :\nabx \bm{\phi}-\bff\cdot\bm{\phi} +\bT:\nabx\bm{\phi}\big] \dx\dt
\\&
+
\int_I\int_\omega\big[\partial_t\eta\partial_t\phi-\partial_t\naby\eta\cdot\naby\phi
-
\naby^2\eta:\naby^2\phi+g\phi\big]\dy\dt;
\end{aligned}
\end{equation}
\item[(e)] For
\begin{align*}
\mathcal{E}_{\mathrm{w}}(\mathrm{data}) :=&
\int_{\Omega_{\eta_{0}}}\mathrm{tr}(\bT_0)\dx
+
\Vert \bu_0\Vert_{L^2(\Omega_{\eta_{0}})}^2 
+
 \Vert\eta_\star\Vert_{L^2(\omega)}^2 
+ 
\Vert\Dely \eta_0\Vert_{L^2(\omega)}^2
\\&+
T\int_{\Omega_{\eta_{0}}}\rho_0 \dx
+
\int_I\Vert \bff\Vert_{L^2(\Oeta)}^2\dt
+
\int_I\Vert g\Vert_{L^2(\omega)}^2\dt,
\end{align*}
the energy inequality
\begin{equation}
\begin{aligned}
\sup_{t\in I}&
\bigg(
\int_{\Oeta}\mathrm{tr}(\bT(t))\dx
+
\Vert \bu(t)\Vert_{L^2(\Oeta)}^2 
+
 \Vert\partial_t \eta(t)\Vert_{L^2(\omega)}^2 
+ 
\Vert\Dely \eta(t)\Vert_{L^2(\omega)}^2
\bigg)
\\&+
\int_I
\int_{\Oeta}\mathrm{tr}(\bT)\dx\dt
+\int_I\Vert \nabx \bu \Vert_{L^2(\Oeta)}^2\dt
+\int_I\Vert\partial_t\naby \eta \Vert_{L^2(\omega)}^2\dt
\\&\lesssim
\mathcal{E}_{\mathrm{w}}(\mathrm{data}) 
\end{aligned}
\end{equation}
holds.
\item[(f)] In addition, we have
\begin{equation}
\begin{aligned}   
\sup_{t\in I}&
\big(\Vert\rho(t)\Vert_{L^2(\Oeta)}^2
+
\Vert\bT(t)\Vert_{L^2(\Oeta)}^2
\big)
+
\int_I\big(\Vert\nabx \rho \Vert_{L^2(\Oeta)}^2
+
\Vert\nabx \bT \Vert_{L^2(\Oeta)}^2\big)\dt
\\&+
\int_I \Vert\bT\Vert_{L^2(\Oeta)}^2\dt 
\lesssim
e^{c\mathcal{E}_{\mathrm{w}}(\mathrm{data}) }
\big((1+T)
\Vert\rho_0 \Vert_{L^2(\Omega_{\eta_0})}^2
+ 
\Vert\bT_0\Vert_{L^2(\Omega_{\eta_0})}^2
\big).
\end{aligned}
\end{equation} 
\end{itemize}
\end{definition}
Even for fixed domains (where formally $\Omega=\Oeta$ and $\eta \equiv 0$ ), the question of the existence of weak solutions to both the 3D incompressible
and the compressible Oldroyd-B model remains a nontrivial open problem. 
In $2$D, however,  the existence of global-in-time weak solutions has been shown in \cite{barrett2011existence} whereas a unique global-in-time strong solution has been shown to exist in \cite{constantin2012note}. Indeed, the closest result to the existence of weak solutions in the $3$D setting is \cite{bathory2021large} where the solute is described by a combination of the Oldroyd-B and the Giesekus models. Unfortunately, no singular limit result exists to show that one can obtain a weak solution of \eqref{divfree}-\eqref{solute} from that constructed in \cite{bathory2021large} for the mixed Oldroyd-B and the Giesekus models. The main obstacle in the construction of a weak solution in this 3D case is the lack of a useful estimate for the extra-stress tensor $\bT$. Indeed, from the basic energy estimate, the only information one derives is that $\mathrm{tr}(\bT) \in   L^{\infty}(I;L^{1}(\Oeta)  )$ and nothing about $\bT$ itself. The latter issue primarily stems from unsuitable interpolation estimates in 3D.

For flexible domains, the existence of a weak solution of \eqref{divfree}--\eqref{interface}, in the sense of Definition \ref{def:weaksolmartFP}, has been shown in \cite{mensah2023weak} for the 2D/1D polymeric fluid-structure system. The paper \cite{mensah2023weak} also includes the construction of a unique global strong solutions
for the 2D/1D version of \eqref{divfree}--\eqref{interface} in analogy with \cite{constantin2012note} for fixed domains. 

Returning to the 3D/2D setting, we now give the precise definition of a strong solution.

%A careful inspection of the proof shows that no critical dimension-dependent estimate (such as interpolation inequalities and Sobolev embeddings) is used in the construction of these weak. Consequently, the same analysis holds in this current  3D/2D setting. Criticality only comes into play once we begin to explore regularity beyond that possessed by weak solutions. A step in this direction is the notion of a strong solution whose precise definition is given as follows.
\begin{definition}[Strong solution]
\label{def:strongSolution}
Let $(\bff, g, \rho_0, \bT_0, \bu_0, \eta_0, \eta_\star)$
be a dataset that satisfies
\begin{equation}
\begin{aligned}
\label{mainDataForAllStrong}
&\bff \in L^2(I;L^{2}_\mathrm{loc}(\mathbb{R}^3)),
\qquad g\in L^2(I;L^{2}(\omega)) ,
\\&
\eta_0 \in W^{3,2}(\omega) \text{ with } \Vert \eta_0 \Vert_{L^\infty( \omega)} < L, \quad \eta_\star \in W^{1,2}(\omega),
\\&\bu_0 \in W^{1,2}_{\divx}(\Omega_{\eta_0} )\text{ is such that }\bu_0 \circ \bm{\varphi}_{\eta_0} =\eta_\star \bn \text{ on } \omega,
\\&
\rho_0\in W^{1,2}(\Omega_{\eta_0}), \quad
\bT_0\in W^{1,2}(\Omega_{\eta_0}),
\\&
\rho_0\geq 0,\,\, \bT_0>0 \quad \text{a.e. in } \Omega_{\eta_0}.
\end{aligned}
\end{equation}
We call 
$(\eta, \bu, p, \rho,\bT)$ 
a \textit{strong solution} of   \eqref{divfree}--\eqref{interface} with dataset $(\bff, g, \rho_0, \bT_0, \bu_0, \eta_0, \eta_\star)$  if:
\begin{itemize}
\item[(a)] the structure-function $\eta$ is such that $
\Vert \eta \Vert_{L^\infty(I \times \omega)} <L$ and
\begin{align*}
\eta \in W^{1,\infty}\big(I;W^{1,2}(\omega)  \big)\cap L^{\infty}\big(I;W^{3,2}(\omega)  \big) \cap  W^{1,2}\big(I; W^{2,2}(\omega)  \big)\cap  W^{2,2}\big(I;L^{2}(\omega)  \big) \cap  L^2\big(I;W^{4,2}(\omega)  \big);
\end{align*}
\item[(b)] the velocity $\bu$ is such that $\bu  \circ \bm{\varphi}_{\eta} =(\partial_t\eta)\bn$ on $I\times \omega$ and
\begin{align*} 
\bu\in  W^{1,2} \big(I; L^2_{\divx}(\Oeta ) \big)\cap L^2\big(I;W^{2,2}(\Oeta)  \big);
\end{align*}
\item[(c)] the pressure $p$ is such that 
\begin{align*}
p\in L^2\big(I;W^{1,2}(\Oeta)  \big);
\end{align*}
\item[(d)] the pair $(\rho,\bT)$ is such that 
\begin{align*}
\rho,\bT \in W^{1,2}\big(I;L^{2}(\Oeta)  \big) \cap L^\infty\big(I;W^{1,2}(\Oeta)  \big)\cap L^2\big(I;W^{2,2}(\Oeta)  \big);
\end{align*}
\item[(e)] equations \eqref{divfree}--\eqref{solute} are satisfied a.e. in spacetime with $\eta(0)=\eta_0$ and $\partial_t\eta(0)=\eta_\star$ a.e. in $\omega$, as well as $\bfu(0)=\bfu_0$, $\rho(0)=\rho_0$ and $\bT(0)=\bT_0$ a.e. in $\Omega_{\eta_0}$.
\end{itemize}
\end{definition}

With this definition in hand, we now state our main result.
\begin{theorem}
\label{thm:MAIN}
For a global weak solution $(\eta,\bu,\rho,\bT)$ of \eqref{divfree}--\eqref{interface} with dataset $(\bff, g, \rho_0, \bT_0, \bu_0, \eta_0, \eta_\star)$  satisfying \eqref{mainDataForAllStrong}, if
\begin{align}
\label{fluid-structureLPScondition}
\eta\in L^\infty(I;C^1(\omega)),\qquad
\bu\in L^{r}(I;L^{s}(\Oeta)),  
\end{align}
holds for $r\in[2,\infty)$ and  $s\in(3,\infty]$ satisfying $2/{r}+3/{s}\leq 1$, then $(\eta,\bu,p,\rho,\bT)$ is a global strong solution of \eqref{divfree}--\eqref{interface} that is unique in class of weak solutions.
\end{theorem}
%{\color{blue}
% Unlike the classical Serrin condition under which a weak solution for the 3D Navier--Stokes equation is a strong solution, the Serrin condition presented in this work is rather of a peculiar mixed type. To clearly explain this peculiarity, we recall that the system we study consists of three distinct subsystems: the solute, solvent, and structure. To construct a strong solution, we impose conditions on the unknowns for each of these three subsystem. Without any regard to the conditions of the solvent and structure subsystem, and just relying on the condition of the solute, we can construct for the first time, a global weak solution for the full mutually-coupled 3D/2D polymeric fluid-structure interaction system. Once this weak solution has been constructed, one only needs to rely only the conditions of the solvent and structure subsystems to ensure that this weak solution is strong in the sense of Definition \ldots\ldots.
%}
%
%***************************************************
 
Let us recall that the classical Ladyzhenskaya--Prodi--Serrin condition for just a solvent modeled by the incompressible Navier--Stokes equation requires that weak solution additionally satisfy $\bu\in L^{r}(I;L^{s}(\Oeta))$ for them to be a strong solution. For a solvent interacting with a viscoelastic structure, a weak solution $(\eta,\bu)$ is expected to satisfy \eqref{fluid-structureLPScondition} for them to become a strong solution. Since finding weak solutions to the $3$D Oldroyld-B is a notorious open problem, one would have certainly expected additional assumptions on $\bT$ for these elusive weak solutions to become strong ones. This turns out not to be the case as stated in Theorem \ref{thm:MAIN}.

\begin{remark}
Although a weak solution $(\eta,\bu,\rho,\bT)$ of \eqref{divfree}--\eqref{interface} is unknown, a weaker result is that a weak solution for the solvent-structure problem can constructed for a given tensor $\overline{\bT}\in L^2(I;W^{1,2}_{\mathrm{loc}}(\mathbb{R}^3))$ on one hand. On the other hand, for a given (strong) velocity $\overline{\bu}\in  W^{1,2} \big(I; L^2_{\divx}(\Oeta ) \big)\cap L^2\big(I;W^{2,2}(\Oeta)  \big)$ (or that  $\overline{\bu}\in L^2\big(I;W^{1,\infty}(\Oeta)  \big)$)  and a shell displacement function $\overline{\eta}\in W^{1,\infty}(I;W^{1,2}(\omega))\cap W^{1,2}(I;W^{2,2}(\omega))$ (or that $\overline{\eta}\in W^{1,8}(I;L^\infty(\omega))$), a weak solution can also be constructed for the solute subproblem. Once an approximation procedure exists for the solvent-structure subproblem, a fixed-point argument can then be used to close the entire system. The gist of these details can be seen in the arguments presented in the next section below.
\end{remark}

\begin{remark}
Just as in \cite[Proposition 1]{constantin2012note},  a solution $\bT$ of \eqref{solute} advected by a velocity field $\bu\in L^2(I;W^{1,\infty}(\Oeta))$ remains strictly positive if it were initially so. Similarly,  a solution $\rho$ of  \eqref{rhoEqu} also remains nonnegative if it were initially nonnegative.
Indeed,  if we test  \eqref{rhoEqu} with the nonpositive part $\rho_-=\min\{0,\rho\}$ of $\rho$, integrate over $\Oeta $ and use the boundary condition \eqref{bddCondSolv} together with Reynold's transport theorem, we obtain
\begin{align*}
\frac{1}{2}\frac{\dd}{\dt}\int_{\Oeta } \vert \rho_-\vert^2 \dx
+\int_{\Oeta } \vert \nabx \rho_-\vert^2\dx =0.
\end{align*} 
Therefore,   it follows that  $ \rho_-=0$ a.e. in $\Oeta$ for any $t\in I$ and thus, $ \rho=\rho_+=\max\{0,\rho\}$. 
\end{remark} 
 
\section{Strong solutions to subproblems}
\label{sec:strongSol}
The initial step in our strategy for constructing a solution involves solving the solvent-structure subproblem and the solute subproblem independently of each other. After that, in the next section, we will use a fixed-point argument to get a local solution to the fully coupled system. The very last section will then involve the extension from a local to a global solution.
\subsection{The solvent-structure subproblem}
In the following, for a given stress tensor $\underline{\bT}$, a given pair of body forces $\bff$ and $g$, we wish to find a strong solution to the following solvent-structure system of equations
\begin{align}
\divx \bu=0, \label{divfreeAlone} 
\\
\partial_t \bu  + (\mathbf{u}\cdot \nabx)\mathbf{u} 
= 
 \delx \bu -\nabx p
+\bff
+
\divx   \underline{\bT}, 
\\
\partial_t^2 \eta - \partial_t\dely \eta + \dely^2 \eta = g - ( \mathbb{S}\bn_\eta )\circ \bm{\varphi}_\eta\cdot\bn \,\det(\naby\bm{\varphi}_\eta),\label{shellAlone}
\end{align}
defined on $I\times\Oeta\subset \mathbb R^{1+3}$ (with \eqref{shellAlone} defined on   $I\times\omega\subset \mathbb R^{1+2}$) where
\begin{align*}
\mathbb{S}=  \nabx \bu +(\nabx \bu)^\top  -p\mathbb{I}+ \underline{\bT}.
\end{align*}
We then complement \eqref{divfreeAlone}--\eqref{shellAlone} with the following initial and interface conditions
\begin{align}
&\eta(0,\cdot)=\eta_0(\cdot), \qquad\partial_t\eta(0,\cdot)=\eta_\star(\cdot) & \text{in }\omega,  
\\  
&\bu(0,\cdot)=\bu_0(\cdot) & \text{in }\Omega_{\eta_0}.
\\
\label{interfaceAlone}
&\bu\circ\bm{\varphi}_\eta=(\partial_t\eta)\bn & \text{on }I\times \omega.
\end{align}
For the setup above, the precise definition of a strong solution is given as follows.
\begin{definition}[Strong solution]
\label{def:strongSolutionAlone}
Let $(\bff, g, \eta_0, \eta_\star, \bu_0, \underline{\bT})$ be a dataset that satisfies
\begin{equation}
\begin{aligned}
\label{datasetAlone}
&\bff \in L^2\big(I; L^2_{\mathrm{loc}}(\mathbb{R}^3 )\big),\quad
g \in L^2\big(I; L^{2}(\omega)\big), \quad
\eta_0 \in W^{3,2}(\omega) \text{ with } \Vert \eta_0 \Vert_{L^\infty( \omega)} < L, 
\\
&\eta_\star \in W^{1,2}(\omega), \quad
\underline{\bT}\in L^2(I;W^{1,2}_{\mathrm{loc}}(\R^3)), \quad
\bu_0\in W^{1,2}_{\mathrm{\divx}}(\Omega_{\eta_0} ) \text{ is such that }\bu_0 \circ \bm{\varphi}_{\eta_0} =\eta_\star \bn \text{ on } \omega.
\end{aligned}
\end{equation} 
We say that
$( \eta, \bu,  p )$
is a \textit{strong solution}  of  \eqref{divfreeAlone}--\eqref{interfaceAlone} with  dataset $(\bff, g, \eta_0, \eta_\star, \bu_0, \underline{\bT})$ if the following holds:
\begin{itemize}
\item[(a)]  the structure displacement $\eta$ is such that $
\Vert \eta \Vert_{L^\infty(I \times \omega)} <L$ and
\begin{align*}
\eta \in W^{1,\infty}\big(I;W^{1,2}(\omega)  \big)\cap L^{\infty}\big(I;W^{3,2}(\omega)  \big) \cap  W^{1,2}\big(I; W^{2,2}(\omega)  \big)\cap  W^{2,2}\big(I;L^{2}(\omega)  \big) \cap  L^{2}\big(I;W^{4,2}(\omega)  \big) ;
\end{align*}
\item[(b)] the velocity $\bu$ is such that $\bu  \circ \bm{\varphi}_{\eta} =(\partial_t\eta)\bn$ on $I\times \omega$ and
\begin{align*} 
\bu\in  W^{1,2} \big(I; L^2_{\divx}(\Oeta ) \big)\cap L^2\big(I;W^{2,2}(\Oeta)  \big);
\end{align*}
\item[(c)] the pressure $p$ is such that 
\begin{align*}
p\in L^2\big(I;W^{1,2}(\Oeta)  \big);
\end{align*}
\item[(d)] the equations \eqref{divfreeAlone}--\eqref{shellAlone} are satisfied a.e. in spacetime with $\eta(0)=\eta_0$ and $\partial_t\eta=\eta_\star$ a.e. in $\omega$ as well as $\bfu(0)=\bfu_0$ a.e. in $\Omega_{\eta_0}$.
\end{itemize}
\end{definition}
\noindent
The existence of a conditionally unique global-in-time strong solution to \eqref{divfreeAlone}--\eqref{interfaceAlone} in the sense of Definition \ref{def:strongSolutionAlone} has recently been shown in \cite[Theorem 1.1]{BMSS}. The statement of the result is: 
\begin{theorem}\label{thm:BMSS}
Let $(\eta,\bu)$ be a weak solution to \eqref{divfreeAlone}--\eqref{interfaceAlone}. Suppose that we
have
\begin{align}\label{eq:regu'0}
\eta\in L^\infty(I;C^{1}(\omega)), \qquad 
\bu\in L^{r}(I;L^{s}(\Omega_\eta))
\end{align}
for any $r\in[2,\infty)$ and $s\in(3,\infty]$ such that $2/{r}+3/{s}\leq1$.
Then $(\eta,\bu)$ is a strong solution to \eqref{divfreeAlone}--\eqref{interfaceAlone}.
Moreover, $(\eta,\bu)$  is unique in the class of weak solutions satisfying the energy inequality.
\end{theorem}
Here, a weak solution satisfies \eqref{divfreeAlone}--\eqref{interfaceAlone} in the sense of distribution and is shown in \cite{lengeler2014weak}. We now proceed to the construction of a strong solution to the solute subproblem.
\subsection{The solute subproblem}
In this section, we construct a strong solution of the solute subproblem for a given flexible domain $\Omega_\zeta$ and a known solenoidal vector field $\bv$. This subproblem is given by:
\begin{align} 
\partial_t \rho+ (\bv\cdot \nabx) \rho
= 
\Delx \rho 
,\label{rhoEquAlone}
\\
\partial_t \bT + (\bv\cdot \nabx) \bT
= 
(\nabx \bv)\bT + \bT(\nabx \bv)^\top - 2(\bT - \rho \mathbb{I})+\Delx \bT \label{soluteSubPro}
\end{align}
on $I\times\Omega_\zeta\subset \mathbb R^{1+3}$ subject to the following initial and boundary conditions
\begin{align} 
&\rho(0,\cdot)=\rho_0(\cdot),\quad\bT(0,\cdot)=\bT_0(\cdot) &\text{in }\Omega_{\zeta(0)},
\label{initialCondSolvSubPro} 
\\
& 
\bn_{\zeta}\cdot\nabx\rho=0,\qquad
\bn_{\zeta}\cdot\nabx\bT=0 &\text{on }I\times\partial\Omega_{\zeta}.
\label{bddCondSolvAlone}
\end{align}
The two unknowns $\rho$ and $\bT$ for the solute component of the polymer fluid are related via the identities
\begin{align*} 
\bT(t, \bx)= \int_{B} f(t,\bx,\bq)\bq\otimes\bq \dq, 
\qquad
\rho(t, \bx)= \int_{B} f(t,\bx,\bq) \dq
\end{align*}
where $f$ solves \eqref{fokkerPlanck}.
We now state the precise definition of what we mean by a strong solution.  
\begin{definition} 
\label{def:strsolmartFP}
Assume that  $(\rho_0,\bT_0, \bv, \zeta)$ satisfies
\begin{equation}
\begin{aligned}
\label{fokkerPlanckDataAlone}
&\rho_0,\bT_0\in  W^{1,2}( \Omega_{\zeta(0)}),
\\&
\rho_{0}\geq 0,\,\, \bT_{0}>0 \quad \text{a.e. in } \Omega_{\zeta(0)} ,
\\&
\bv\in W^{1,2} \big(I; L^2_{\divx}(\Omega_\zeta ) \big)\cap L^2\big(I;W^{2,2}(\Omega_\zeta )  \big),
\\
&\zeta\in W^{1,\infty}\big(I;W^{1,2}(\omega)  \big)\cap L^{\infty}\big(I;W^{3,2}(\omega)  \big) \cap  W^{1,2}\big(I; W^{2,2}(\omega)  \big)\cap  W^{2,2}\big(I;L^{2}(\omega)  \big)
,
\\& \bv  \circ \bm{\varphi}_{\zeta} =(\partial_t\zeta)\bn
\quad \text{on }I \times \omega,  \quad\|\zeta\|_{L^\infty(I\times\omega)}<L.
\end{aligned}
\end{equation}
We call
$(\rho,\bT)$
a \textit{strong solution} of   \eqref{rhoEquAlone}-\eqref{bddCondSolvAlone} with dataset $(\rho_0,\bT_0, \bv, \zeta)$ if 
\begin{itemize}
\item[(a)] $(\rho,\bT)$ satisfies
\begin{align*}
\rho,\bT&\in   W^{1,2} \big(I; L^2(\Omega_\zeta ) \big)\cap L^\infty\big(I;W^{1,2}(\Omega_\zeta )  \big)
\cap L^2\big(I;W^{2,2}(\Omega_\zeta )  \big);
\end{align*}
\item[(b)]  equations \eqref{rhoEquAlone}--\eqref{soluteSubPro} are satisfied a.e. in spacetime with $\rho(0)=\rho_0$ and $\bT(0)=\bT_0$ a.e. in $\Omega_{\zeta(0)}$.
\end{itemize}
\end{definition}
\noindent We now formulate our result on the existence of a unique strong solution of \eqref{rhoEquAlone}-\eqref{bddCondSolvAlone}.
\begin{theorem}\label{thm:mainFP}
For a dataset  $(\rho_0,\bT_0, \bv, \zeta)$ satisfying  \eqref{fokkerPlanckDataAlone}, there exists a unique strong solution  $(\rho,\bT)$ of  \eqref{rhoEquAlone}-\eqref{bddCondSolvAlone}, in the sense of Definition \ref{def:strsolmartFP},
 such that
\begin{equation}
\begin{aligned} 
\int_I\big(\Vert \partial_t \rho\Vert_{L^2(\Ozeta)}^2
&+
\Vert \partial_t \bT\Vert_{L^2(\Ozeta)}^2
\big)\dt
+
\sup_{t\in I} \big(\Vert \rho(t)\Vert_{W^{1,2}(\Ozeta)}^2
+
\Vert \bT(t)\Vert_{W^{1,2}(\Ozeta)}^2
\big)
\\&
+\int_I\big(\Vert  \rho\Vert_{W^{2,2}(\Ozeta)}
+
\Vert  \bT\Vert_{W^{2,2}(\Ozeta)}\big)
\\&\lesssim
\Vert  \rho_0\Vert_{W^{1,2}(\Omega_{\zeta(0)})}
+
\Vert  \bT_0\Vert_{W^{1,2}(\Omega_{\zeta(0)})}
+
\int_I\big(\Vert\partial_t\zeta  \Vert_{W^{2,2}(\omega)}^2+\Vert\bv\Vert_{W^{2,2}(\Ozeta)}^2\big)\dt
%\\& \quad
%+
%T \Vert\rho_0  \Vert_{L^2(\Omega_{\zeta(0)})}^2 
%+ 
%\int_I  
%\Vert\bT \Vert_{L^{s_2}(\Ozeta)}^{r_2}\dt 
\label{eq:thm:mainFP}
\end{aligned}
\end{equation}  
holds.
\end{theorem}
\noindent 
Since \eqref{rhoEquAlone} and \eqref{soluteSubPro} are dissipative and bilinear,  a strong solution of  \eqref{rhoEquAlone}-\eqref{bddCondSolvAlone}, and in particular the bound \eqref{eq:thm:mainFP},  is directly obtained by way of a limit to a Galerkin approximation. We therefore proceed with just the formal computation.
\begin{proof}
First of all, if we test \eqref{rhoEquAlone} with $\rho$ and use Reynold's transport theorem and the boundary condition, we obtain,
\begin{align*}
\frac{1}{2}
\int_I  \frac{\mathrm{d}}{\dt}
\Vert\rho \Vert_{L^2(\Ozeta)}^2 \dt 
+
\int_I \Vert\nabx\rho\Vert_{L^2(\Ozeta)}^2\dt
&= 
0
\end{align*}
so that
\begin{align*}
\frac{1}{2}
\sup_{t\in I} 
\Vert\rho(t) \Vert_{L^2(\Ozeta)}^2 
+
\int_I \Vert\nabx\rho\Vert_{L^2(\Ozeta)}^2\dt
&= 
\frac{1}{2}
\Vert\rho_0\Vert_{L^2(\Omega_{\zeta(0)})}^2.
\end{align*}
If we also test \eqref{soluteSubPro} with  $\bT$, we obtain  
\begin{align*}
\frac{1}{2}
\int_I  \frac{\mathrm{d}}{\dt}
\Vert\bT \Vert_{L^2(\Ozeta)}^2 \dt 
+
\int_I\big(\Vert\nabx\bT\Vert_{L^2(\Ozeta)}^2
+
2 \Vert\bT\Vert_{L^2(\Ozeta)}^2\big)\dt
&= 
\int_I\int_{\Ozeta}
[(\nabx \bv )\bT  + \bT (\nabx \bv )^\top]:\bT \dx\dt
\\&
+2\int_I\int_{\Ozeta} \rho  \mathrm{tr}(\bT )\dx\dt.
\end{align*}
On one hand, we have
\begin{align*}
2\int_I\int_{\Ozeta} \rho  \mathrm{tr}(\bT )\dx\dt
\leq 
\int_I\big(\Vert\rho\Vert_{L^2(\Ozeta)}^2 
+
 \Vert\mathrm{tr}(\bT)\Vert_{L^2(\Ozeta)}^2\big) \dt
\leq
T\Vert\rho_0\Vert_{L^2(\Omega_{\zeta(0)})}^2
+
\int_I\Vert \bT \Vert_{L^2(\Ozeta)}^2 \dt
\end{align*}
and on the other hand, by interpolation 
\begin{align*}
\int_I\int_{\Ozeta}
[(\nabx \bv )\bT  + \bT (\nabx \bv )^\top]:\bT \dx\dt
&
\lesssim
\int_I 
\Vert\nabx \bv\Vert_{L^2(\Ozeta)} \Vert\bT\Vert_{L^{6}(\Ozeta)}
\Vert\bT \Vert_{L^{3}(\Ozeta)}\dt
\\& 
\lesssim
\int_I 
\Vert\nabx \bv\Vert_{L^2(\Ozeta)}
\Vert\nabx \bT \Vert_{L^2(\Ozeta)}^{3/2}
\Vert\bT \Vert_{L^{2}(\Ozeta)}^{1/2}\dt 
\\
&\leq  
\delta 
\int_I  
\Vert\nabx \bT \Vert_{L^2(\Ozeta)}^2\dt 
+
c
\int_I  
\Vert\nabx \bv \Vert_{L^2(\Ozeta)}^4
\Vert\bT \Vert_{L^{2}(\Ozeta)}^2\dt  
\end{align*} 
holds for any $\delta>0$.
Thus, it follows from Gr\"onwall's lemma that
\begin{equation}
\begin{aligned}
\label{strgEstFP3}  
\sup_{t\in I}&
\Vert\bT(t)\Vert_{L^2(\Ozeta)}^2 
+\int_I\big( 
\Vert\nabx \bT \Vert_{L^2(\Ozeta)}^2
+
\Vert\bT\Vert_{L^2(\Ozeta)}^2
\big)\dt
\\&
\lesssim  
e^{\int_I  
\Vert\nabx \bv \Vert_{L^2(\Ozeta)}^4\dt}
\bigg[
\Vert\bT_0\Vert_{L^2(\Omega_{\zeta(0)})}^2
+
T \Vert\rho_0  \Vert_{L^2(\Omega_{\zeta(0)})}^2 
\bigg] 
\\&
\lesssim  
e^{\Vert\bv \Vert_{W^{1,2}(I;L^2(\Ozeta))}^2
\int_I  
\Vert\bv \Vert_{W^{2,2}(\Ozeta)}^2\dt}
\bigg[
\Vert\bT_0\Vert_{L^2(\Omega_{\zeta(0)})}^2
+
T \Vert\rho_0  \Vert_{L^2(\Omega_{\zeta(0)})}^2 
\bigg] 
\end{aligned}
\end{equation}
since by interpolation
\begin{align*}
\int_I  
\Vert\nabx \bv \Vert_{L^2(\Ozeta)}^4\dt 
\lesssim
\int_I  
\Vert\bv \Vert_{L^2(\Ozeta)}^2
\Vert\bv \Vert_{W^{2,2}(\Ozeta)}^2\dt 
\lesssim 
\Vert\bv \Vert_{W^{1,2}(I;L^2(\Ozeta))}^2
\int_I  
\Vert\bv \Vert_{W^{2,2}(\Ozeta)}^2\dt.
\end{align*}
We now test \eqref{rhoEquAlone}  with $\Delx \rho$. This yields
\begin{equation}
\begin{aligned}  
\label{spaceRegRho0}
\int_I  \frac{\mathrm{d}}{\dt}
\Vert\nabx \rho\Vert_{L^2(\Ozeta)}^2 \dt 
+
\int_I\Vert\Delx \rho\Vert_{L^2(\Ozeta)}^2\dt
&=
\int_I
\int_{\Ozeta}((\bv\cdot \nabx) \rho)\Delx\rho \dx\dt
\\&+
\frac{1}{2}
\int_I
\int_{\partial\Ozeta}(\partial_t\zeta\bn)\circ\bm{\varphi}_\zeta^{-1}\cdot\bn_\zeta \vert\nabx\rho\vert^2 \dd\mathcal{H}^2\dt
\\&=:I_1+I_2
\end{aligned}
\end{equation}
where
\begin{align*}
I_1\leq \delta
\int_I\Vert\Delx \rho\Vert_{L^2(\Ozeta)}^2\dt
+
c(\delta)
\int_I\Vert\bv\Vert_{W^{2,2}(\Ozeta)}^2\Vert\nabx \rho \Vert_{L^2(\Ozeta)}^2  \dt
\end{align*}
for any $\delta>0$. Now note that by interpolation and the trace theorem, 
\begin{align*}
\Vert\nabx \rho \Vert_{L^2(\partial\Ozeta)}^2
&\lesssim
\Vert\nabx \rho \Vert_{W^{1,2}(\Ozeta)}
\Vert\nabx \rho \Vert_{W^{1/4,2}(\partial\Ozeta)}
\\
&\lesssim
\Vert\nabx \rho \Vert_{W^{1,2}(\Ozeta)}
\Vert\nabx \rho \Vert_{W^{3/4,2}(\Ozeta)}
\\
&\lesssim
\Vert\nabx \rho \Vert_{W^{1,2}(\Ozeta)}
\Vert\nabx \rho \Vert_{L^{2}(\Ozeta)}^{1/4}
\Vert\nabx \rho \Vert_{W^{1,2}(\Ozeta)}^{3/4}.
\end{align*}
Thus by using $\zeta\in L^\infty(I;W^{1,\infty}(\omega))$ and $\partial_t\zeta\in L^\infty(I;W^{1,2}(\omega))$, we also obtain
\begin{align*}
I_2&\lesssim
\int_I\Vert\nabx \rho \Vert_{L^2(\partial\Ozeta)}^2\Vert(\partial_t\zeta\bn)\circ\bm{\varphi}_\zeta^{-1}\cdot\bn_\zeta  \Vert_{L^\infty(\partial\Ozeta)} \dt
\\
&\lesssim
\int_I\Vert  \rho \Vert_{W^{2,2}(\Ozeta)}^{7/4}  \Vert\nabx \rho \Vert_{L^{2}(\Ozeta)}^{1/4}
\Vert\naby\zeta  \Vert_{L^{\infty}(\omega)}
\Vert\partial_t\zeta  \Vert_{W^{5/4,2}(\omega)} \dt
\\
&\lesssim 
\int_I\Vert  \rho \Vert_{W^{2,2}(\Ozeta)}^{7/4}  \Vert\nabx \rho \Vert_{L^{2}(\Ozeta)}^{1/4}
\Vert\partial_t\zeta  \Vert_{W^{1,2}(\omega)} ^{3/4}
\Vert\partial_t\zeta  \Vert_{W^{2,2}(\omega)}^{1/4}  \dt
\\
&\leq
\delta
\int_I\Vert  \rho \Vert_{W^{2,2}(\Ozeta)}^2+c(\delta)\int_I  \Vert\nabx \rho \Vert_{L^{2}(\Ozeta)}^2\Vert\partial_t\zeta  \Vert_{W^{2,2}(\omega)}^2  \dt
\end{align*}
or any $\delta>0$.
Consequently, it follows from Gr\"onwall's lemma that
\begin{equation}
\begin{aligned}
\label{spaceRegRho1}
\sup_{t\in I}\Vert \nabx\rho(t)\Vert_{L^2(\Ozeta)}&+\int_I\Vert \Delx\rho\Vert_{L^2(\Ozeta)}\dt
\\&\lesssim
\Vert \nabx\rho_0\Vert_{L^2(\Omega_{\zeta(0)})}
+
\int_I\big(\Vert\partial_t\zeta  \Vert_{W^{2,2}(\omega)}^2+\Vert\bv\Vert_{W^{2,2}(\Ozeta)}^2\big)\dt.
\end{aligned}
\end{equation}
If we also test \eqref{soluteSubPro} with $\Delx \bT$, we obtain 
\begin{equation}
\begin{aligned}  
\label{spaceRegT0}
\int_I  \frac{\mathrm{d}}{\dt}
\Vert\nabx \bT\Vert_{L^2(\Ozeta)}^2\dt 
+
\int_I\Vert\Delx \bT\Vert_{L^2(\Ozeta)}^2\dt
&=
\int_I
\int_{\Ozeta}((\bv\cdot \nabx) \bT)\Delx\bT \dx\dt
\\&+
\frac{1}{2}
\int_I
\int_{\partial\Ozeta}(\partial_t\zeta\bn)\circ\bm{\varphi}_\zeta^{-1}\cdot\bn_\zeta \vert\nabx\bT\vert^2 \dd\mathcal{H}^2\dt
\\&+
2\int_I
\int_{\Ozeta}(\bT-\rho\mathbb{I})\Delx\bT \dx\dt
\\&
-\int_I
\int_{\Ozeta}((\nabx \bv)\bT + \bT(\nabx \bv)^\top)\Delx\bT \dx\dt
\\&=:J_1+J_2+J_3+J_4.
\end{aligned}
\end{equation}
The terms $J_1$ and $J_2$ can be treated as $I_1$ and $I_2$ above. By using \eqref{strgEstFP3}, we obtain
\begin{align*}
J_3\leq &\delta
\int_I\Vert\Delx \bT\Vert_{L^2(\Ozeta)}^2\dt
+
cT\Big[ \Vert\bT_0\Vert_{L^2(\Omega_{\zeta(0)})}^2
+
T \Vert\rho_0  \Vert_{L^2(\Omega_{\zeta(0)})}^2 \Big] 
\end{align*}
for any $\delta>0$ and by Ladyzhenskaya's inequality and  \eqref{strgEstFP3},
\begin{align*}
J_4&\lesssim
\int_I  
\Vert\nabx \bv\Vert_{L^2(\Ozeta)}^{1/4}\Vert \bv\Vert_{W^{2,2}(\Ozeta)}^{3/4}\Vert \bT\Vert_{L^2(\Ozeta)}^{1/4}\Vert\nabx \bT\Vert_{L^2(\Ozeta)}^{3/4}\Vert\Delx \bT\Vert_{L^2(\Ozeta)}
\dt
\\
\leq &\delta
\int_I\Vert\Delx \bT\Vert_{L^2(\Ozeta)}^2\dt
+
c\int_I \Vert \bv\Vert_{W^{2,2}(\Ozeta)}^2\Vert\nabx \bT\Vert_{L^2(\Ozeta)}^2
\dt
\\&+
c\sup_{t\in I} \Vert \bT(t)\Vert_{L^{2}(\Ozeta)}^2\int_I\Vert\nabx \bv\Vert_{L^2(\Ozeta)}^2
\dt
\end{align*}
where by \eqref{strgEstFP3} and the regularity of the dataset,
\begin{align*}
\sup_{t\in I} \Vert \bT\Vert_{L^{2}(\Ozeta)}^2
&\int_I\Vert\nabx \bv\Vert_{L^2(\Ozeta)}^2
\dt
\lesssim
\sup_{t\in I} \Vert \bT\Vert_{L^{2}(\Ozeta)}^2
\\&
\lesssim
 \Vert\bT_0\Vert_{L^2(\Omega_{\zeta(0)})}^2
+
T \Vert\rho_0  \Vert_{L^2(\Omega_{\zeta(0)})}^2.
\end{align*}
Thus, it follows from Gr\"onwall's lemma that
\begin{equation}
\begin{aligned}
\label{spaceRegT1}
\sup_{t\in I}&\Vert \nabx\bT(t)\Vert_{L^2(\Ozeta)}
+
\int_I\Vert \Delx\bT\Vert_{L^2(\Ozeta)}\dt
\\&\lesssim
\Vert \nabx\bT_0\Vert_{L^2(\Omega_{\zeta(0)})}
+
\int_I\big(\Vert\partial_t\zeta  \Vert_{W^{2,2}(\omega)}^2+\Vert\bv\Vert_{W^{2,2}(\Ozeta)}^2\big)\dt  
\\&\quad
+
\Vert\bT_0\Vert_{L^2(\Omega_{\zeta(0)})}^2
+
T \Vert\rho_0  \Vert_{L^2(\Omega_{\zeta(0)})}^2  
.
\end{aligned}
\end{equation}
To obtain regularity in time, we test \eqref{rhoEqu}  with $\partial_t \rho$. This yields
\begin{equation}
\begin{aligned}
\label{timeRegRho0}
\int_I\Vert \partial_t \rho\Vert_{L^2(\Ozeta)}^2\dt
&+
\int_I\frac{\dd}{\dt}\Vert \nabx \rho\Vert_{L^2(\Ozeta)}^2\dt
\\&=
\frac{1}{2}
\int_I
\int_{\partial\Ozeta}(\partial_t\zeta\bn)\circ\bm{\varphi}_\zeta^{-1}\cdot\bn_\zeta \vert\nabx\rho\vert^2 \dd\mathcal{H}^2\dt
\\&
-\int_I\int_{\Ozeta}(\bv\cdot \nabx) \rho\partial_t\rho\dx\dt
\\
&\leq
c
\int_I\Vert  \rho \Vert_{W^{2,2}(\Ozeta)}^2
+
c\int_I  \Vert\nabx \rho \Vert_{L^{2}(\Ozeta)}^2\Vert\partial_t\zeta  \Vert_{W^{2,2}(\omega)}^2  \dt
\\
&+c(\delta)\int_I \Vert\nabx \rho \Vert_{L^{2}(\Ozeta)}^2 \Vert\bv \Vert_{W^{2,2}(\Ozeta)}^2  \dt
+
\delta
\int_I\Vert  \partial_t\rho \Vert_{L^{2}(\Ozeta)}^2
\end{aligned}
\end{equation}
for any $\delta>0$. Note the estimate for the boundary term done earlier in \eqref{spaceRegRho0}.
By using \eqref{spaceRegRho1}, it follows from \eqref{timeRegRho0} that
\begin{equation}
\begin{aligned}
\label{timeRegRho1}
\int_I\Vert \partial_t \rho\Vert_{L^2(\Ozeta)}^2\dt
&+
\sup_{t\in I} \Vert \nabx \rho(t)\Vert_{L^2(\Ozeta)}^2
\\&\lesssim
\Vert \nabx\rho_0\Vert_{L^2(\Omega_{\zeta(0)})}
+
\int_I\big(\Vert\partial_t\zeta  \Vert_{W^{2,2}(\omega)}^2+\Vert\bv\Vert_{W^{2,2}(\Ozeta)}^2\big)\dt.
\end{aligned}
\end{equation}
Now, we note that  (compare with the estimate for $J_3$ in \eqref{spaceRegT0})
\begin{align*}
\int_I\int_{\Ozeta}(\bT - &\rho \mathbb{I}):\partial_t\bT\dx\dt
\leq \delta
\int_I\Vert\partial_t \bT\Vert_{L^2(\Ozeta)}^2\dt
 +
cT\Big[  \Vert\bT_0\Vert_{L^2(\Omega_{\zeta(0)})}^2
+
T \Vert\rho_0  \Vert_{L^2(\Omega_{\zeta(0)})}^2    \Big]
\end{align*}
and (compare with the estimate for $J_4$ in \eqref{spaceRegT0})
\begin{align*}
\int_I\int_{\Ozeta}
[(\nabx \bv)\bT &+ \bT(\nabx \bv)^\top]:\partial_t\bT\dx\dt
\\&\leq \delta
\int_I\Vert\partial_t \bT\Vert_{L^2(\Ozeta)}^2\dt
+
c\int_I \Vert \bv\Vert_{W^{2,2}(\Ozeta)}^2\Vert\nabx \bT\Vert_{L^2(\Ozeta)}^2 
\\&
+
 \Vert\bT_0\Vert_{L^2(\Omega_{\zeta(0)})}^2
+
T \Vert\rho_0  \Vert_{L^2(\Omega_{\zeta(0)})}^2  
.
\end{align*}
Therefore, by testing \eqref{solute} with $\partial_t\bT$, we obtain
\begin{equation}
\begin{aligned}
\label{timeRegT1}
\int_I\Vert \partial_t \bT\Vert_{L^2(\Ozeta)}^2\dt
&+
\sup_{t\in I} \Vert \nabx \bT(t)\Vert_{L^2(\Ozeta)}^2
\\&\lesssim
\Vert \nabx\bT_0\Vert_{L^2(\Omega_{\zeta(0)})}
+
\int_I\big(\Vert\partial_t\zeta  \Vert_{W^{2,2}(\omega)}^2+\Vert\bv\Vert_{W^{2,2}(\Ozeta)}^2\big)\dt 
\\&\quad
+
 \Vert\bT_0\Vert_{L^2(\Omega_{\zeta(0)})}^2
+
T \Vert\rho_0  \Vert_{L^2(\Omega_{\zeta(0)})}^2  
.
\end{aligned}
\end{equation}
If we now combine \eqref{spaceRegRho1}, \eqref{spaceRegT1} \eqref{timeRegRho1} and \eqref{timeRegT1}, we obtain the estimate
\begin{equation}
\begin{aligned} \nonumber
\int_I\big(\Vert \partial_t \rho\Vert_{L^2(\Ozeta)}^2
&+
\Vert \partial_t \bT\Vert_{L^2(\Ozeta)}^2
\big)\dt
+
\sup_{t\in I} \big(\Vert \rho(t)\Vert_{W^{1,2}(\Ozeta)}^2
+
\Vert \bT(t)\Vert_{W^{1,2}(\Ozeta)}^2
\big)
\\&
+\int_I\big(\Vert  \rho\Vert_{W^{2,2}(\Ozeta)}
+
\Vert  \bT\Vert_{W^{2,2}(\Ozeta)}\big)
\\&\lesssim
\Vert  \rho_0\Vert_{W^{1,2}(\Omega_{\zeta(0)})}
+
\Vert  \bT_0\Vert_{W^{1,2}(\Omega_{\zeta(0)})}
+
\int_I\big(\Vert\partial_t\zeta  \Vert_{W^{2,2}(\omega)}^2+\Vert\bv\Vert_{W^{2,2}(\Ozeta)}^2\big)\dt
%\\&
%\quad+
%T \Vert\rho_0  \Vert_{L^2(\Omega_{\zeta(0)})}^2 
%+ 
%\int_I  
%\Vert\bT \Vert_{L^{s_2}(\Ozeta)}^{r_2}\dt
.
\end{aligned}
\end{equation}  
as desired.
\end{proof} 

\section{Local strong solution via Fixed-Point}
To construct a local-in-time solution for the fully coupled system, we resort to a fixed-point argument. The existence of such a fixed-point will follow from closedness and contraction properties for a suitable map. More precisely, to show the closedness property, we consider the mapping $\mathtt{T}=\mathtt{T}_1\circ\mathtt{T}_2$ with
\begin{align*}
\mathtt{T}(\underline{\rho}, \underline{\bT})=({\rho}, {\bT}), \qquad \mathtt{T}_2(\underline{\rho}, \underline{\bT})=(\eta,\bu,p), \qquad \mathtt{T}_1(\eta,\bu,p)=( {\rho}, {\bT})
\end{align*}
defined on the space
\begin{align*}
X_\eta&:=W^{1,2} (I_*;L^{2}(\Oeta)  )
\cap
L^\infty (I_*;W^{1,2}(\Oeta)  )
\cap L^{2} (I_*;W^{2,2}(\Oeta)  ), 
\end{align*}
equipped with its canonical norm $\Vert \cdot\Vert_{X_\eta}$ and where $I_*$ is the local time yet to be determined. We now consider the subset
\begin{align*}
B_R:=\big\{  (\underline{\rho}, \underline{\bT})\in X_\eta \otimes X_\eta \,:\,\Vert (\underline{\rho}, \underline{\bT})\Vert_{X_\eta \otimes X_\eta} \leq R  \big\}.
\end{align*}
and show that $\mathtt{T}:X_\eta\otimes X_\eta\rightarrow X_\eta\otimes X_\eta$ maps $B_R$ into $B_R$, i.e., for any $(\underline{\rho}, \underline{\bT})\in   B_R$, we have that 
\begin{align*}
\Vert (\rho, \bT) \Vert_{X_\eta \otimes X_\eta}^2=\Vert \mathtt{T}(\underline{\rho}, \underline{\bT})\Vert_{X_\eta \otimes X_\eta}^2=\Vert \mathtt{T}_1\circ \mathtt{T}_2(\underline{\rho}, \underline{\bT})\Vert_{X_\eta\otimes X_\eta}^2=\Vert \mathtt{T}_1(\eta,\bu, p) \Vert_{X_\eta\otimes X_\eta}^2 \leq R^2.
\end{align*}
Indeed, for $ \underline{\bT}\in X_\eta$,  if we let $( \eta, \bu,  p )$ be the (conditional) unique strong solution  of  \eqref{divfreeAlone}--\eqref{interfaceAlone} with data $(\bff, g, \eta_0, \eta_\star, \bu_0, \underline{\bT})$ as shown in  Theorem \ref{thm:BMSS}, then by \cite[(4.5)]{BMSS},  $( \eta, \bu,  p )$ satisfies 
\begin{equation}
\begin{aligned} 
\int_{I_*}\big(\Vert\partial_t\eta  \Vert_{W^{2,2}(\omega)}^2+\Vert\bu\Vert_{W^{2,2}(\Oeta)}^2\big)\dt
&\lesssim
\Vert \bu_0 \Vert_{W^{1,2}(\eta_0)}^2+ \Vert\eta_0\Vert_{W^{3,2}(\omega)}^2+ \Vert \eta_\star\Vert_{W^{1,2}(\omega)}^2
\\& 
+\int_{I_*}\big(\Vert g\Vert_{L^{2}(\omega)}^2
+
\Vert \bff\Vert_{L^{2}(\Oeta)}^2
+
\Vert \underline{\bT}\Vert_{W^{1,2}(\Oeta)}^2\big)\dt.
\label{eq:thm:mainFPxy}
\end{aligned}
\end{equation} 
On the other hand, for 
\begin{align*}
(\eta,\bu)\in&  
W^{1,\infty}\big(I_*;W^{1,2}(\omega)  \big)\cap L^{\infty}\big(I_*;W^{3,2}(\omega)  \big) \cap  W^{1,2}\big(I_*; W^{2,2}(\omega)  \big)\cap  W^{2,2}\big(I_*;L^{2}(\omega)  \big)
\cap  L^2\big(I;W^{4,2}(\omega)  \big)
\\&\otimes
W^{1,2} \big(I_*; L^2_{\divx}(\Oeta ) \big)\cap L^2\big(I_*;W^{2,2}(\Oeta )  \big), 
\end{align*}
let $(\rho,\bT)$ be the unique strong solution of \eqref{rhoEquAlone}-\eqref{bddCondSolvAlone} with dataset $(\rho_0,\bT_0, \bu, \eta)$ as shown in Theorem \ref{thm:mainFP}. 
As shown in \eqref{eq:thm:mainFP}, $(\rho,\bT)$ will satisfy
\begin{equation}
\begin{aligned} 
\Vert (\rho, \bT) \Vert_{X_\eta \otimes X_\eta}^2
&\lesssim
\Vert  \rho_0\Vert_{W^{1,2}(\Omega_{\eta_0})}
+
\Vert  \bT_0\Vert_{W^{1,2}(\Omega_{\eta_0})}
+
\int_{I_*}\big(\Vert\partial_t\eta  \Vert_{W^{2,2}(\omega)}^2+\Vert\bu\Vert_{W^{2,2}(\Oeta)}^2\big)\dt
.
\label{eq:thm:mainFPx}
\end{aligned}
\end{equation} 
Given the regularity of  the dataset and the fact that $ \underline{\bT}\in X_\eta$, for a large enough $R>0$ and $T_*>0$ small enough, we obtain by substituting \eqref{eq:thm:mainFPxy} into \eqref{eq:thm:mainFPx},
\begin{align*}
\Vert (\rho, \bT) \Vert_{X_\eta \otimes X_\eta}^2\leq R^2
\end{align*}
by substituting \eqref{eq:thm:mainFPxy} into \eqref{eq:thm:mainFPx}. Thus $\mathtt{T}:B_R\rightarrow B_R$.
\\
Now, for the contraction property, we consider the superset $Y_\eta \supseteq X_\eta$ defined by 
\begin{align*} 
Y_\eta&:=  L^{\infty} (I_*;L^2(\Oeta ) )
\cap 
L^2(I_*;W^{1,2}(\Oeta) ),
\end{align*}
and equipped with its canonical norm $\Vert \cdot\Vert_{Y_\eta}$, and show that
\begin{equation}
\begin{aligned}
\label{contrEst00}
\Vert \mathtt{T}( \underline{\rho}_1- \underline{\rho}_2,  \underline{\bT}_1-\underline{\bT}_2) \Vert_{Y_{\eta_1}\otimes Y_{\eta_1}}^2
&\leq
\tfrac{1}{2}
\Vert ( \underline{\rho}_1- \underline{\rho}_2,  \underline{\bT}_1-\underline{\bT}_2)\Vert_{Y_{\eta_1}\otimes Y_{\eta_1}}^2.
\end{aligned}
\end{equation} 
hold for any two pair of strong solutions $(\underline{\rho}_i, \underline{\bT}_i)\in X_{\eta_1}\otimes X_{\eta_1}$, $i=1,2$ for the solute subproblem \eqref{rhoEquAlone}-\eqref{bddCondSolvAlone}
with dataset $(\rho_0,\bT_0, \bu_i, \eta_i)$, $i=1,2$, respectively. To show \eqref{contrEst00}, we note that since the fluid domain depends on the deformation of the shell, we have to transform one solution, say $\bT_2$, to the domain of $\bT_1$ to get a difference estimate. For this, we set  $\overline{\bT}_2=\bT_2\circ \bm{\Psi}_{\eta_2-\eta_1}$, $\overline{\bu}_2=\bu_2\circ \bm{\Psi}_{\eta_2-\eta_1}$, $\overline{\rho}_2=\rho_2\circ \bm{\Psi}_{\eta_2-\eta_1}$ and obtain (see \cite{mensah2023weak} for further details)
\begin{align*}
\partial_t\overline{\bT}_2
+
\overline{\bu}_2 \cdot\nabx \overline{\bT}_2 
&=
\nabx\overline{\bu}_2\overline{\bT}_2
+
\overline{\bT}_2 (\nabx\overline{\bu}_2)^\top
-
2
(\overline{\bT}_2-\overline{\rho}_2\mathbb{I})
+
\Delx\overline{\bT}_2
\\&
-
\divx(\nabx \overline{\bT}_2(\mathbb{I}-\mathbb{A}_{\eta_2-\eta_1}))
+
\mathbb{H}_{\eta_2-\eta_1}(\overline{\rho}_2,\overline{\bu}_2,\overline{\bT}_2)
\end{align*}
defined on ${I_*} \times \Omega_{\eta_1}$ where
\begin{align*}
\mathbb{H}_{\eta_2-\eta_1}(\overline{\rho}_2,\overline{\bu}_2,\overline{\bT}_2)
&=
(1-J_{\eta_2-\eta_1})\partial_t\overline{\bT}_2
-
J_{\eta_2-\eta_1} \nabx \overline{\bT}_2\cdot\partial_t\bm{\Psi}_{\eta_2-\eta_1}^{-1}\circ \bm{\Psi}_{\eta_2-\eta_1}
+
\nabx\overline{\bu}_2(\mathbb{B}_{\eta_2-\eta_1}-\mathbb{I})\overline{\bT}_2
\\&+
\overline{\bT}_2(\mathbb{B}_{\eta_2-\eta_1}-\mathbb{I})^\top (\nabx\overline{\bu}_2)^\top
+
\overline{\bu}_2 \cdot\nabx \overline{\bT}_2 (\mathbb{I}-\mathbb{B}_{\eta_2-\eta_1})
+
2
(1-J_{\eta_2-\eta_1}) (\overline{\bT}_2-\overline{\rho}_2\mathbb{I}).
\end{align*}
We now consider the difference
\begin{align*}
\bT_{12}:=\bT_1-\overline{\bT}_2,
\quad
\bu_{12}=\bu_1-\overline{\bu}_2, 
\quad
\rho_{12}=\rho_1-\overline{\rho}_2,
\quad
\eta_{12}=\eta_1-\eta_2.
\end{align*}
and find that $\bT_{12}$ solves
\begin{equation}
\begin{aligned}
\label{diffEquaSigma}
\partial_t \bT_{12}+ \mathbf{u}_1\cdot \nabx \bT_{12}
&=
\nabx \bu_1\bT_{12} + \bT_{12}(\nabx \bu_1)^\top - 2(\bT_{12} - \rho_{12} \mathbb{I})
+\Delx \bT_{12}
\\&
+
\nabx \bu_{12}\overline{\bT}_2 + \overline{\bT}_2(\nabx \bu_{12})^\top 
- \mathbf{u}_{12}\cdot \nabx \overline{\bT}_2
\\&+
\divx(\nabx \overline{\bT}_2(\mathbb{I}-\mathbb{A}_{-\eta_{12}}))
-
\mathbb{H}_{-\eta_{12}}(\overline{\rho}_2,\overline{\bu}_2,\overline{\bT}_2)
\end{aligned}
\end{equation}
on ${I_*} \times \Omega_{\eta_1}$ with identically zero initial condition.
If we now test \eqref{diffEquaSigma} with $\bT_{12}$, then for $t\in {I_*}$, we obtain
\begin{equation}
\begin{aligned}
\frac{1}{2}&
\frac{\dd}{\dt}\Vert \bT_{12}\Vert_{L^2(\Omega_{\eta_1})}^2
+
\Vert \nabx\bT_{12}\Vert_{L^2(\Omega_{\eta_1})}^2
+ 
\Vert \bT_{12}\Vert_{L^2(\Omega_{\eta_1})}^2
\leq
\Vert \rho_{12}\Vert_{L^2(\Omega_{\eta_1})}^2
\\&
+
\int_{\partial\Omega_{\eta_1}}
(\bn_{\eta_1}\cdot\nabx)\bT_{12}:\bT_{12}
\dd\mathcal{H}^1
+
2\int_{\Omega_{\eta_1}}\vert\nabx\bu_1\vert
\vert \bT_{12}\vert^2\dx
\\&
+
\int_{\Omega_{\eta_1}}
\big[
\nabx \bu_{12}\overline{\bT}_2 + \overline{\bT}_2(\nabx \bu_{12})^\top 
\big]:\bT_{12}\dx
-
\int_{\Omega_{\eta_1}}
\mathbf{u}_{12}\cdot \nabx \overline{\bT}_2
 :\bT_{12} \dx
\\&
+
\int_{\Omega_{\eta_1}} 
\divx(\nabx \overline{\bT}_2(\mathbb{I}-\mathbb{A}_{-\eta_{12}}))
:\bT_{12} \dx
-
\int_{\Omega_{\eta_1}}
\mathbb{H}_{-\eta_{12}}(\overline{\rho}_2,\overline{\bu}_2,\overline{\bT}_2) :\bT_{12} \dx
\\&
=: \Vert \rho_{12}\Vert_{L^2(\Omega_{\eta_1})}^2+I_1+\ldots+I_6
\end{aligned}
\end{equation} 
Now note that  $\bn_{\eta_1}\cdot\nabx\bT_1=0$ on ${I_*}\times \partial\Omega_{\eta_1}$ and that
\begin{align*}
\Vert \bT_{12}\Vert_{L^2(\partial\Omega_{\eta_1})}
\lesssim
\Vert \bT_{12}\Vert_{W^{1/4,2}(\partial\Omega_{\eta_1})}
\lesssim
\Vert \bT_{12}\Vert_{W^{3/4,2}(\Omega_{\eta_1})}
\lesssim
\Vert \bT_{12}\Vert_{L^{2}(\Omega_{\eta_1})}^{1/4}
\Vert \nabx \bT_{12}\Vert_{L^{2}(\Omega_{\eta_1})}^{3/4}.
\end{align*}
Thus, it follows that for any $t\in \overline{I_*}$,
\begin{align*}
\int_0^t
I_1
\dt'
&
\lesssim
\int_0^t
\Vert
 \overline{\bT}_2
\Vert_{W^{2,2}(\Omega_{\eta_1})}
\Vert 
\bT_{12}
\Vert_{L^2(\Omega_{\eta_1})}^{1/4}
\Vert
\nabx
\bT_{12}
\Vert_{L^2(\Omega_{\eta_1})}^{3/4} \dt'
\\&\leq
\delta
\int_0^t
\Vert
 \overline{\bT}_2
\Vert_{W^{2,2}(\Omega_{\eta_1})}^2 \dt'
+
\delta
\int_0^t
\Vert \nabx\bT_{12}\Vert_{L^2(\Omega_{\eta_1})}^2\dt'
+
c
\int_0^t
\Vert 
\bT_{12}
\Vert_{L^2(\Omega_{\eta_1})}^2\dt'
\\&\leq
\delta
+
\delta
\int_0^t
\Vert \nabx\bT_{12}\Vert_{L^2(\Omega_{\eta_1})}^2\dt'
+
c
\int_0^t
\Vert 
\bT_{12}
\Vert_{L^2(\Omega_{\eta_1})}^2\dt'
\end{align*} 
Next, we use interpolation to obtain
\begin{align*}
\int_0^t
I_2\dt'
&\lesssim 
\int_0^t
\Vert\nabx \bu_1\Vert_{L^6(\Omega_{\eta_1})}  \Vert\bT_{12}\Vert_{L^2(\Omega_{\eta_1})}
\Vert\bT_{12} \Vert_{L^{3}(\Omega_{\eta_1})} \dt'
\\&\lesssim 
\int_0^t
\Vert  \bu_1\Vert_{W^{2,2}(\Omega_{\eta_1})}
 \Vert\bT_{12}\Vert_{L^2(\Omega_{\eta_1})}^{3/2}
\Vert\nabx\bT_{12} \Vert_{L^{2}(\Omega_{\eta_1})}^{1/2} \dt'
\\
&\leq
\delta  
\int_0^t
\Vert\nabx \bT_{12} \Vert_{L^2(\Omega_{\eta_1})}^2\dt'
+
c\int_0^t
\big(1+
\Vert  \bu_1\Vert_{W^{2,2}(\Omega_{\eta_1})}^{2}
\big)
\Vert\bT_{12} \Vert_{L^2(\Omega_{\eta_1})}^{2}\dt'
\end{align*} 
for any $\delta>0$.
Next, given that the embedding $W^{1,2}(\Omega_{\eta_1})\hookrightarrow L^{6}(\Omega_{\eta_1})$ is continuous, it follows from interpolation that
\begin{equation}
\begin{aligned}
\int_0^t
I_3\dt
&\lesssim
\int_0^t
\Vert\nabx\bu_{12}\Vert_{L^2(\Omega_{\eta_1})}
\Vert \overline{\bT}_2\Vert_{L^{6}(\Omega_{\eta_1})}
\Vert \bT_{12}\Vert_{L^3(\Omega_{\eta_1})}\dt'
\\
&\lesssim
\int_0^t
\Vert\nabx\bu_{12}\Vert_{L^2(\Omega_{\eta_1})}
\Vert \overline{\bT}_2\Vert_{W^{1,2}(\Omega_{\eta_1})}
\Vert \bT_{12}\Vert_{L^2(\Omega_{\eta_1})}^{1/2}
\Vert \nabx\bT_{12}\Vert_{L^2(\Omega_{\eta_1})}^{1/2} 
\dt'
\\
&\lesssim
\int_0^t
\Vert\nabx\bu_{12}\Vert_{L^2(\Omega_{\eta_1})}
\Vert \overline{\bT}_2\Vert_{L^{2}(\Omega_{\eta_1})}^{1/2}
\Vert \overline{\bT}_2\Vert_{W^{2,2}(\Omega_{\eta_1})}^{1/2}
\Vert \bT_{12}\Vert_{L^2(\Omega_{\eta_1})}^{1/2}
\Vert \nabx\bT_{12}\Vert_{L^2(\Omega_{\eta_1})}^{1/2} 
\dt'
\\
&\leq
\delta
\int_0^t
\Vert \nabx\bu_{12}\Vert_{L^2(\Omega_{\eta_1})}^2\dt'
+
\delta
\int_0^t
\Vert \nabx\bT_{12}\Vert_{L^2(\Omega_{\eta_1})}^2\dt'
+
c
\int_0^t 
\Vert \overline{\bT}_2\Vert_{W^{2,2}(\Omega_{\eta_1})}^2
\Vert \bT_{12}\Vert_{L^2(\Omega_{\eta_1})}^2 \dt'
\end{aligned}
\end{equation} 
for any $\delta>0$ and where we have used the fact that $\Vert \overline{\bT}_2\Vert_{L^{2}(\Omega_{\eta_1})}^2$ is essentially bounded in the last step. Similarly,
\begin{equation}
\begin{aligned}
\int_0^t 
I_4\dt'
&\lesssim
\int_0^t 
\Vert  \bu_{12}\Vert_{L^{6}(\Omega_{\eta_1})}
\Vert \nabx\overline{\bT}_2\Vert_{L^{2}(\Omega_{\eta_1})}
\Vert \bT_{12}\Vert_{L^2(\Omega_{\eta_1})}^{1/2}
\Vert\nabx \bT_{12}\Vert_{L^2(\Omega_{\eta_1})}
^{1/2}\dt'
\\&\leq
\delta
\int_0^t 
\Vert \nabx\bu_{12}\Vert_{L^2(\Omega_{\eta_1})}^2\dt'
+
\delta
\int_0^t 
\Vert \nabx\bT_{12}\Vert_{L^2(\Omega_{\eta_1})}^2
\dt'
+
c
\int_0^t  
\Vert \overline{\bT}_2\Vert_{W^{2,2}(\Omega_{\eta_1})}^2
\Vert \bT_{12}\Vert_{L^2(\Omega_{\eta_1})}^2\dt'.
\end{aligned}
\end{equation} 
Next, we rewrite $I_5$ as
\begin{equation}
\begin{aligned} \nonumber
\int_{\partial\Omega_{\eta_1}}
\bn_{\eta_1}\cdot\nabx \overline{\bT}_2(\mathbb{I}-\mathbb{A}_{-\eta_{12}})
:\bT_{12} \dx
 -
\int_{\Omega_{\eta_1}}
 \nabx \overline{\bT}_2(\mathbb{I}-\mathbb{A}_{-\eta_{12}}):
:\nabx\bT_{12} \dx
\end{aligned}
\end{equation}
so by the trace theorem and the fact that
$\mathbb{I}-\mathbb{A}_{-\eta_{12}}\sim -\naby\eta_{12}$ holds in norm,
\begin{equation}
\begin{aligned}\nonumber
\int_{\partial\Omega_{\eta_1}}
\bn_{\eta_1}\cdot\nabx \overline{\bT}_2&(\mathbb{I}-\mathbb{A}_{-\eta_{12}})
:\bT_{12} \dx
\lesssim
\Vert
\nabx \overline{\bT}_2
\Vert_{L^4(\partial\Omega_{\eta_1})}
\Vert
\mathbb{I}-\mathbb{A}_{-\eta_{12}}
\Vert_{L^4(\partial\Omega_{\eta_1})}
\Vert
\bT_{12}
\Vert_{L^2(\partial\Omega_{\eta_1})}
\\&
\lesssim
\Vert
\nabx \overline{\bT}_2
\Vert_{W^{1/2,2}( \partial\Omega_{\eta_1})}
\Vert
\eta_{12}
\Vert_{W^{1,4}(\omega)}
\Vert 
\bT_{12}
\Vert_{W^{3/4,2}(\Omega_{\eta_1})} 
\\&
\lesssim
\Vert
  \overline{\bT}_2
\Vert_{W^{2,2}( \Omega_{\eta_1})}
\Vert
\eta_{12}
\Vert_{W^{2,2}(\omega)}
\Vert 
\bT_{12}
\Vert_{L^2(\Omega_{\eta_1})}^{1/4}
\Vert
\nabx
\bT_{12}
\Vert_{L^2(\Omega_{\eta_1})}^{3/4}
\\&\leq
\delta
\Vert \nabx\bT_{12}\Vert_{L^2(\Omega_{\eta_1})}^2
+
\delta
\Vert 
\bT_{12}
\Vert_{L^2(\Omega_{\eta_1})}^2
+
c
\Vert  \overline{\bT}_2
\Vert_{W^{2,2}(\Omega_{\eta_1})}^2
\Vert  \eta_{12} \Vert_{W^{2,2}(\omega )}^2
\end{aligned}
\end{equation}
whereas
\begin{equation}
\begin{aligned}\nonumber
\int_{\Omega_{\eta_1}}
 \nabx \overline{\bT}_2&(\mathbb{I}-\mathbb{A}_{-\eta_{12}}):
:\nabx\bT_{12} \dx 
\\&
\leq
\delta
\Vert \nabx\bT_{12}\Vert_{L^2(\Omega_{\eta_1})}^2
+
c
\Vert  \overline{\bT}_2
\Vert_{W^{2,2}(\Omega_{\eta_1})}^2
\Vert  \eta_{12} \Vert_{W^{2,2}(\omega )}^2.
\end{aligned}
\end{equation}
It follows that
\begin{equation}
\begin{aligned}
\int_0^t 
I_5\dt'
\leq
\delta
\int_0^t 
\Vert \nabx\bT_{12}\Vert_{L^2(\Omega_{\eta_1})}^2
\dt'
+
\delta
\int_0^t 
\Vert 
\bT_{12}
\Vert_{L^2(\Omega_{\eta_1})}^2\dt'
+
c
\int_0^t 
\Vert  \overline{\bT}_2
\Vert_{W^{2,2}(\Omega_{\eta_1})}^2
\Vert  \eta_{12} \Vert_{W^{2,2}(\omega )}^2
\dt'
\end{aligned}
\end{equation}
To deal with $I_6$ we first expand it as follows
\begin{equation}
\begin{aligned}
I_6 
&=
\int_{\Omega_{\eta_1}}
(1-J_{-\eta_{12}})\partial_t\overline{\bT}_2
 :\bT_{12} \dx
\\&-
\int_{\Omega_{\eta_1}}
J_{-\eta_{12}} \nabx \overline{\bT}_2\cdot\partial_t\bm{\Psi}_{-\eta_{12}}^{-1}\circ \bm{\Psi}_{-\eta_{12}}
 :\bT_{12} \dx
\\&+
\int_{\Omega_{\eta_1}}
\nabx\overline{\bu}_2(\mathbb{B}_{-\eta_{12}}-\mathbb{I})\overline{\bT}_2
 :\bT_{12} \dx
 \\&+
\int_{\Omega_{\eta_1}}
\overline{\bT}_2(\mathbb{B}_{-\eta_{12}}-\mathbb{I})^\top (\nabx\overline{\bu}_2)^\top
:\bT_{12} \dx
 \\&+
\int_{\Omega_{\eta_1}}
\overline{\bu}_2 \cdot\nabx \overline{\bT}_2 (\mathbb{I}-\mathbb{B}_{-\eta_{12}})
 :\bT_{12} \dx
  \\&+2
\int_{\Omega_{\eta_1}}
(1-J_{-\eta_{12}}) (\overline{\bT}_2-\overline{\rho}_2\mathbb{I}) :\bT_{12} \dx
\\&=:I_6^1+\ldots+I_6^6.
\end{aligned}
\end{equation} 
Then we have, by interpolation 
\begin{align*}
I_6^1&\lesssim
\Vert  \eta_{12}\Vert_{W^{1,6}(\omega )}
\Vert \partial_t\overline{\bT}_2\Vert_{L^{2}(\Omega_{\eta_1})}
\Vert  \bT_{12}\Vert_{L^3(\Omega_{\eta_1})}
\\
&\lesssim
\Vert  \eta_{12}\Vert_{W^{2,2}(\omega )}
\Vert \partial_t\overline{\bT}_2\Vert_{L^{2}(\Omega_{\eta_1})}
\Vert  \bT_{12}\Vert_{L^{2}(\Omega_{\eta_1})}^{1/2}
\Vert \nabx \bT_{12}\Vert_{L^{2}(\Omega_{\eta_1})}^{1/2} 
\\
&\leq
\delta
\Vert \nabx \bT_{12}\Vert_{L^{2}(\Omega_{\eta_1})}^2
+\delta
\Vert  \bT_{12}\Vert_{L^{2}(\Omega_{\eta_1})}^2
+
c
\Vert  \eta_{12}\Vert_{W^{2,2}(\omega )}^2
\Vert \partial_t\overline{\bT}_2\Vert_{L^{2}(\Omega_{\eta_1})}^2.
\end{align*}
For $I_6^2$, we use  interpolation to obtain
\begin{align*}
I_6^2 
&\lesssim
\Vert  \partial_t\eta_{12}\Vert_{W^{1,2}(\omega )}
\Vert \overline{\bT}_2\Vert_{W^{1,6}(\Omega_{\eta_1})}
\Vert \bT_{12}\Vert_{L^3(\Omega_{\eta_1})}
\\
&\lesssim
\Vert  \partial_t\eta_{12}\Vert_{W^{1,2}(\omega )} 
\Vert \overline{\bT}_2\Vert_{W^{2,2}(\Omega_{\eta_1})} 
\Vert \bT_{12}\Vert_{L^2(\Omega_{\eta_1})}^{1/2}
\Vert \nabx\bT_{12}\Vert_{L^2(\Omega_{\eta_1})}^{1/2} 
\\
&\leq
\delta
\Vert \nabx\bT_{12}\Vert_{L^2(\Omega_{\eta_1})}^2
+
\delta 
\Vert \bT_{12}\Vert_{L^2(\Omega_{\eta_1})}^2
+
c
\Vert  \partial_t\eta_{12}\Vert_{W^{1,2}(\omega )}^2
\Vert \overline{\bT}_2\Vert_{W^{1,2}(\Omega_{\eta_1})}^2
.
\end{align*}
Next, we have
\begin{align*}
I_6^3 +I_6^4 
&\lesssim  
\Vert \nabx\overline{\bu}_2\Vert_{L^{6}(\Omega_{\eta_1})}
\Vert \eta_{12}\Vert_{W^{1,6}(\omega )}
\Vert \overline{\bT}_2\Vert_{L^{6}(\Omega_{\eta_1})}
\Vert \bT_{12}\Vert_{L^2(\Omega_{\eta_1})}
\\&
\lesssim
\Vert \eta_{12}\Vert_{W^{2,2}(\omega )}^2
\Vert \overline{\bT}_2\Vert_{W^{1,2}(\Omega_{\eta_1})}^2
+
\Vert \bT_{12}\Vert_{L^2(\Omega_{\eta_1})}^2
\Vert \overline{\bu}_2\Vert_{W^{2,2}(\Omega_{\eta_1})}^2
\end{align*}
and similarly,
\begin{align*}
I_6^5
&\lesssim
\Vert \eta_{12}\Vert_{W^{2,2}(\omega )}^2
\Vert \overline{\bT}_2\Vert_{W^{2,2}(\Omega_{\eta_1})}^2
+
\Vert \bT_{12}\Vert_{L^2(\Omega_{\eta_1})}^2
\Vert \overline{\bu}_2\Vert_{W^{1,2}(\Omega_{\eta_1})}^2,
\\&I_6^6
\leq
\delta
\Vert \bT_{12}\Vert_{L^2(\Omega_{\eta_1})}^2
+
c
\Vert  \eta_{12} \Vert_{W^{2,2}(\omega )}^2
\big(\Vert \overline{\bT}_2\Vert_{W^{1,2}(\Omega_{\eta_1})}^2
+
\Vert \overline{\rho}_2\Vert_{W^{1,2}(\Omega_{\eta_1})}^2\big)
.
\end{align*}
Thus, it follows that
\begin{equation}
\begin{aligned}
\int_0^t 
I_6\dt'
\leq&
\delta
\int_0^t 
\Vert \nabx \bT_{12}\Vert_{L^{2}(\Omega_{\eta_1})}^2\dt'
+
\delta
\int_0^t 
\Vert  \bT_{12}\Vert_{L^{2}(\Omega_{\eta_1})}^2\dt'
+
c\int_0^t 
\Vert \bT_{12}\Vert_{L^2(\Omega_{\eta_1})}^2
\Vert \overline{\bu}_2\Vert_{W^{2,2}(\Omega_{\eta_1})}^2\dt'
\\&+
c\sup_{t\in I_*}
\Vert  \eta_{12}(t)\Vert_{W^{2,2}(\omega )}^2
+
c
\int_0^t 
\Vert  \partial_{t'}\eta_{12}\Vert_{W^{1,2}(\omega )}^2\dt'
.
\end{aligned}
\end{equation}
where we used the estimate
\begin{align*}
\sup_{t\in I_*}
\Vert \overline{\bT}_2(t)\Vert_{W^{1,2}(\Omega_{\eta_1})}^2
+
\int_0^t 
\big(
\Vert \partial_{t'}\overline{\bT}_2\Vert_{L^{2}(\Omega_{\eta_1})}^2
+
\Vert \overline{\bT}_2\Vert_{W^{2,2}(\Omega_{\eta_1})}^2
+
\Vert \overline{\rho}_2\Vert_{W^{1,2}(\Omega_{\eta_1})}^2\big)\dt'
\lesssim 1.
\end{align*}
If we combine the estimates for $I_1,\ldots,I_6$, we obtain
\begin{equation}
\begin{aligned} \Vert \bT_{12}(t)&\Vert_{L^2(\Omega_{\eta_1})}^2
+
\int_0^t\big(
\Vert \nabx\bT_{12}\Vert_{L^2(\Omega_{\eta_1})}^2
+ 
\Vert \bT_{12}\Vert_{L^2(\Omega_{\eta_1})}^2\big)\dt'
\\
\lesssim
&\delta
 +
\sup_{t\in I_*}
\Vert  \eta_{12}(t)\Vert_{W^{2,2}(\omega )}^2
\\&+
\int_0^t
\big(
\Vert \rho_{12}\Vert_{L^2(\Omega_{\eta_1})}^2
+ 
\Vert\nabx \bu_{12}\Vert_{L^{2}(\Omega_{\eta_1})}^2
+ 
\Vert  \partial_{t'}\eta_{12}\Vert_{W^{1,2}(\omega )}^2
\big)\dt'
\\&+
 \int_0^t
\Vert \bT_{12}\Vert_{L^2(\Omega_{\eta_1})}^2
\big(1+
\Vert  \bu_1\Vert_{W^{2,2}(\Omega_{\eta_1})}^{2}
+
\Vert \overline{\bu}_2\Vert_{W^{2,2}(\Omega_{\eta_1})}^2
+ 
\Vert \overline{\bT}_2\Vert_{W^{2,2}(\Omega_{\eta_1})}^2
\big)\dt'.
\end{aligned}
\end{equation} 
for any $t\in \overline{I_*}$ so that by Gr\"onwall's lemma, we obtain
\begin{equation}
\begin{aligned}
\sup_{t\in I_*}&\Vert \bT_{12}(t)\Vert_{L^2(\Omega_{\eta_1})}^2
+
\int_{I_*}\big(
\Vert \nabx\bT_{12}\Vert_{L^2(\Omega_{\eta_1})}^2
+ 
\Vert \bT_{12}\Vert_{L^2(\Omega_{\eta_1})}^2\big)\dt
\\
\lesssim
&e^{c\int_{I_*} \big(1+
\Vert  \bu_1\Vert_{W^{2,2}(\Omega_{\eta_1})}^{2}+
\Vert \overline{\bu}_2\Vert_{W^{2,2}(\Omega_{\eta_1})}^2
+ 
\Vert \overline{\bT}_2\Vert_{W^{2,2}(\Omega_{\eta_1})}^2
\big)\dt'} 
\\&\times\bigg[
\delta 
+
\sup_{t \in I_*}
\Vert  \eta_{12}(t)\Vert_{W^{2,2}(\omega )}^2 
 +
\int_{I_*}
\big(
\Vert \rho_{12}\Vert_{L^2(\Omega_{\eta_1})}^2
+  
\Vert\nabx \bu_{12}\Vert_{L^{2}(\Omega_{\eta_1})}^2   
+ 
\Vert  \partial_t\eta_{12}\Vert_{W^{1,2}(\omega )}^2
\big)
\dt \bigg]
\end{aligned}
\end{equation} 
Similarly, the difference of two strong solutions of \eqref{rhoEquAlone} satisfies
\begin{equation}
\begin{aligned}
\sup_{t\in {I_*}}&
\Vert \rho_{12}(t)\Vert_{L^2(\Omega_{\eta_1})}^2
+
\int_{I_*}
\Vert \nabx\rho_{12}\Vert_{L^2(\Omega_{\eta_1})}^2
\dt 
\lesssim
e^{c\int_{I_*} \big(1
+ 
\Vert \overline{\bT}_2\Vert_{W^{2,2}(\Omega_{\eta_1})}^2
\big)\dt'} 
\\&\times\bigg[\delta 
+
\sup_{t\in {I_*}}
\Vert  \eta_{12}(t) \Vert_{W^{2,2}(\omega )}^2 
 +
\int_{I_*}
\big(
\Vert \nabx  \bu_{12}\Vert_{L^{2}(\Omega_{\eta_1})}^2+
\Vert  \partial_t\eta_{12}\Vert_{W^{1,2}(\omega )}^2\big)
\dt
\bigg]
\end{aligned}
\end{equation} 
for any $\delta>0$.
Combining the two estimates above therefore yields
\begin{equation}
\begin{aligned}
\label{contrEst0}
\sup_{t\in {I_*}}&
\big(\Vert \rho_{12}(t)\Vert_{L^2(\Omega_{\eta_1})}^2
+
\Vert \bT_{12}(t)\Vert_{L^2(\Omega_{\eta_1})}^2
\big)
+
\int_{I_*}
\big(\Vert \nabx\rho_{12}\Vert_{L^2(\Omega_{\eta_1})}^2
+
\Vert \nabx\bT_{12}\Vert_{L^2(\Omega_{\eta_1})}^2
\big)
\dt 
\\&\lesssim
(1+T_*)\bigg[\delta
+
\sup_{t\in {I_*}}
\Vert  \eta_{12} \Vert_{W^{2,2}(\omega )}^2 
+
\int_{I_*}
\big(
\Vert \nabx \bu_{12}\Vert_{L^{2}(\Omega_{\eta_1})}^2+
\Vert  \partial_t\eta_{12}\Vert_{W^{1,2}(\omega )}^2\big)
\dt 
\bigg]
\end{aligned}
\end{equation} 
for any $\delta>0$.
Now,  let consider two strong solutions $( \eta_i, \bu_i,  p_i )$, $i=1,2$  of  \eqref{divfreeAlone}--\eqref{interfaceAlone} with data $(\bff, g, \eta_0, \eta_\star, \bu_0, \underline{\bT}_i)$, respectively. The existence of these solutions is shown in  Theorem \ref{thm:BMSS} under the Ladyzhenskaya--Prodi--Serrin condition \cite{lady1967uniqueness, prodi1959un, Serrin1963TheIV}
For 
\begin{align*}
\underline{\bT}_{12}:=\underline{\bT}_1-\underline{\overline{\bT}}_2,
\quad
\bu_{12}=\bu_1-\overline{\bu}_2,  
\quad
\eta_{12}=\eta_1-\eta_2,
\end{align*}
where $\underline{\overline{\bT}}_2:=\underline{\bT}_2\circ\bm{\Psi}_{\eta_2-\eta_1}$, it follows from \cite[Remark 5.2]{BMSS} that
\begin{align*}
\sup_{t\in {I_*}}
\Vert  \eta_{12} \Vert_{W^{2,2}(\omega )}^2
+
\int_{I_*}
\big(
\Vert  \nabx\bu_{12}\Vert_{L^{2}(\Omega_{\eta_1})}^2+
\Vert  \partial_t\eta_{12}\Vert_{W^{1,2}(\omega )}^2\big)
\dt
&\lesssim
\int_{I_*}
\Vert  \underline{\bT}_{12}\Vert_{L^{2}(\Omega_{\eta_1})}^2
\dt
\\&\lesssim
T_*
\Vert ( \underline{\rho}_{12},  \underline{\bT}_{12})\Vert_{Y_{\eta_1}\otimes Y_{\eta_1}}^2.
\end{align*}
Inserting into \eqref{contrEst0} then yields
\begin{equation}
\begin{aligned}
\label{contrEst1}
\Vert (\rho_{12}, \bT_{12})\Vert_{Y_{\eta_1}\otimes Y_{\eta_1}}^2
&\lesssim
(1+T_*)\bigg[
\delta
+
T_*
\Vert ( \underline{\rho}_{12},  \underline{\bT}_{12})\Vert_{Y_{\eta_1}\otimes Y_{\eta_1}}^2
\bigg].
\end{aligned}
\end{equation} 
Since a contraction map holds trivially when $\Vert ( \underline{\rho}_{12},  \underline{\bT}_{12})\Vert_{Y_{\eta_1}\otimes Y_{\eta_1}}^2=0$, we may assume that $\Vert ( \underline{\rho}_{12},  \underline{\bT}_{12})\Vert_{Y_{\eta_1}\otimes Y_{\eta_1}}^2$ is strictly positive and bounded the sum of the individual terms $\rho_1,\underline{\rho}_2,\bT_1$ and $\underline{\bT}_2$ in $Y_{\eta_1}$. In this case, we can divide and multiply the $\delta$-term in \eqref{contrEst1} by $\Vert (\rho_{12}, \bT_{12})\Vert_{Y_{\eta_1}\otimes Y_{\eta_1}}^2$ such that for $\delta>0$ and $T_*>0$ chosen small enough, we obtain
\begin{equation}
\begin{aligned}
\label{contrEst2}
\Vert (\rho_{12}, \bT_{12})\Vert_{Y_{\eta_1}\otimes Y_{\eta_1}}^2
&\leq
\tfrac{1}{2}
\Vert ( \underline{\rho}_{12},  \underline{\bT}_{12})\Vert_{Y_{\eta_1}\otimes Y_{\eta_1}}^2.
\end{aligned}
\end{equation} 
The existence of the desired fixed point now follows. 

 \section{Global estimates} 
\label{sec:Apriori}
The goal of this section is to show that the local strong solution constructed in the immediate section above actually holds globally over the whole time interval $I=(0,T)$. This will follow from the following global estimate result. 
\begin{proposition}
\label{prop:strongEst}
Let $r \in[2,\infty)$ and $s \in(3,\infty]$ be such that $2/{r}+3/{s}\leq 1$ hold.
Then any strong solution $(\eta, \bu, p,\rho,\bT)$ 
 of   \eqref{divfree}--\eqref{interface} with dataset $(\bff, g, \rho_0, \bT_0, \bu_0, \eta_0, \eta_\star)$ satisfies
 \begin{equation}
\begin{aligned} 
\int_I&\big(\Vert \partial_t \rho\Vert_{L^2(\Oeta)}^2
+
\Vert \partial_t \bT\Vert_{L^2(\Oeta)}^2
+
\Vert \partial_t \bu\Vert_{L^2(\Oeta)}^2
+
\Vert \partial_t^2\eta\Vert_{L^{2}(\omega)}^2
\big)\dt
\\&+
\sup_{t\in I} \big(\Vert \rho(t)\Vert_{W^{1,2}(\Oeta)}^2
+
\Vert \bT(t)\Vert_{W^{1,2}(\Oeta)}^2
+
\Vert \bu(t)\Vert_{W^{1,2}(\Oeta)}^2
+
\Vert \partial_t\eta(t)\Vert_{W^{1,2}(\omega)}^2
+
\Vert  \eta(t)\Vert_{W^{3,2}(\omega)}^2
\big)
\\&
+\int_I\big(\Vert  \rho\Vert_{W^{2,2}(\Oeta)}
+
\Vert  \bT\Vert_{W^{2,2}(\Oeta)}
+
\Vert  \bu\Vert_{W^{2,2}(\Oeta)}
+
\Vert  \nabx p\Vert_{L^{2}(\Oeta)}
+
\Vert \partial_t\eta\Vert_{W^{2,2}(\omega)}^2
+
\Vert \eta\Vert_{W^{4,2}(\omega)}^2
\big)
\\&\lesssim
\Vert  \rho_0\Vert_{W^{1,2}(\Omega_{\eta_0})}
+
\Vert  \bT_0\Vert_{W^{1,2}(\Omega_{\eta_0})}
+
\Vert  \bu_0\Vert_{W^{1,2}(\Omega_{\eta_0})}
+
 \Vert\eta_\star\Vert_{W^{1,2}(\omega)}^2 
+ 
\Vert\eta_0\Vert_{W^{3,2}(\omega)}^2
\\&
+
\int_I\big(\Vert \bff\Vert_{L^2(\Oeta)}^2
+
\Vert g\Vert_{L^{2}(\omega)}^2\big)\dt.
\label{eq:thm:mainFP00}
\end{aligned}
\end{equation} 
 \end{proposition}
 \begin{proof} 
Take $(\bm{\phi},\phi)=(\bu,\partial_t\eta)$ in \eqref{weakFluidStrut} and use Reynold's transport theorem to obtain 
\begin{equation}
\begin{aligned}
\label{enerFormalFluiStru}
\frac{1}{2}\int_I  \frac{\mathrm{d}}{\dt}\Big(\Vert \bu\Vert_{L^2(\Oeta)}^2 
&+
 \Vert\partial_t \eta\Vert_{L^2(\omega)}^2 
+ 
\Vert\dely\eta\Vert_{L^2(\omega)}^2
\Big)\dt
+
 \int_I\Big(\Vert \nabx \bu\Vert_{L^2(\Oeta)}^2
 +
 \Vert\partial_t\naby \eta\Vert_{L^2(\omega)}^2\Big)\dt
\\&
=-\int_I\int_{\Oeta}\bT :\nabx \bu \dx\dt
+
\int_I\int_{\Oeta}\bff \cdot \bu \dx\dt
+
\int_I \int_\omega g\partial_t \eta\dy\dt
\end{aligned}
\end{equation}
Next, we take the trace in \eqref{solute}, integrate and use \eqref{bddCondSolv}  and the relation $\mathrm{tr}(\mathbb{A}\mathbb{B}^\top)=\mathbb{A}:\mathbb{B}$ which holds for all $\mathbb{A},\mathbb{B}\in \mathbb{R}^{d\times d}$ to obtain
\begin{equation}
\begin{aligned}
\label{freeEnergyEst}
\frac{1}{2}\int_I\frac{\dd}{\dt}\int_{\Oeta} \mathrm{tr}(\bT)\dx\dt
+
  \int_I
  \int_{\Oeta}
  \mathrm{tr}(\bT)\dx\dt 
&=
\int_I
\int_{\Oeta}
 \bT:\nabx \bu \dx \dt
+3
\int_I
\int_{\Oeta}\rho \dx\dt. 
\end{aligned}
\end{equation}
If we also integrate \eqref{rhoEqu} over space and use \eqref{bddCondSolv}, we obtain
\begin{align*}
\int_{\Oeta}\rho \dx=\int_{\Omega_{\eta_{0}}}\rho_0 \dx
\end{align*} 
whereas, if we set $ \psi=\rho$ in \eqref{weakRhoEq} and use Reynold's transport theorem, we obtain
\begin{align*}  
\frac{1}{2}
\int_I  \frac{\mathrm{d}}{\dt}
\Vert\rho \Vert_{L^2(\Oeta)}^2 \dt 
&+
\int_I\Vert\nabx\rho\Vert_{L^2(\Oeta)}^2\dt
=0.
\end{align*}
Combining the equations above with the following inequalities
\begin{align*}
&\int_I\int_{\Oeta}\bff \cdot \bu \dx\dt
\leq
c
\int_I\Vert \bff\Vert_{L^2(\Oeta)}^2\dt
+
\frac{1}{4}
\sup_{t\in I}\Vert\bu(t)\Vert_{L^2(\Oeta)}^2 ,
\\&
\int_I \int_\omega g\partial_t \eta\dy\dt
\leq
c
\int_I\Vert g\Vert_{L^2(\omega)}^2\dt
+
\frac{1}{4}
\sup_{t\in I}\Vert\partial_t \eta(t)\Vert_{L^2(\omega)}^2
\end{align*}
yields
\begin{equation}
\begin{aligned}
\label{apriori1}
\sup_{t\in I}&
\bigg(  
\Vert \bu(t)\Vert_{L^2(\Oeta)}^2 
+
 \Vert\partial_t \eta(t)\Vert_{L^2(\omega)}^2 
+ 
\Vert\dely\eta(t)\Vert_{L^2(\omega)}^2
+
\Vert\rho(t)\Vert_{L^2(\Oeta)}^2
 + 
\int_{\Oeta}\mathrm{tr}(\bT(t))\dx
\bigg)
\\&+
\int_I
\int_{\Oeta}\mathrm{tr}(\bT)\dx\dt
+\int_I
\big(
\Vert\nabx \rho \Vert_{L^2(\Oeta)}^2
+
\Vert \nabx \bu \Vert_{L^2(\Oeta)}^2
+\Vert\partial_t\naby \eta \Vert_{L^2(\omega)}^2\big)\dt 
\\&\lesssim
\mathcal{E}_{\mathrm{w}}(\mathrm{data}) 
\end{aligned}
\end{equation} 
where
\begin{equation}
\begin{aligned}
\mathcal{E}_{\mathrm{w}}(\mathrm{data})
:=&
\int_{\Omega_{\eta_{0}}}\mathrm{tr}(\bT_0)\dx
+
\Vert \bu_0\Vert_{L^2(\Omega_{\eta_{0}})}^2 
+
 \Vert\eta_\star\Vert_{L^2(\omega)}^2 
+ 
\Vert\dely\eta_0\Vert_{L^2(\omega)}^2
\\&
+
\Vert\rho_0 \Vert_{L^2(\Omega_{\eta_0})}^2
+
T
\int_{\Omega_{\eta_{0}}}\rho_0 \dx
+
\int_I\Vert \bff\Vert_{L^2(\Oeta)}^2\dt
+
\int_I\Vert g\Vert_{L^2(\omega)}^2\dt.
\end{aligned}
\end{equation}
Next, just as in \eqref{strgEstFP3}, we obtain 
\begin{equation}
\begin{aligned}
\label{strgEstFP3xx}  
\sup_{t\in I}&
\Vert\bT(t)\Vert_{L^2(\Oeta)}^2 
+\int_I\big( 
\Vert\nabx \bT \Vert_{L^2(\Oeta)}^2
+
\Vert\bT\Vert_{L^2(\Oeta)}^2
\big)\dt
\\&
\lesssim  
\Vert\bT_0\Vert_{L^2(\Omega_{\eta_0})}^2
+
T \Vert\rho_0  \Vert_{L^2(\Omega_{\eta_{0}})}^2  
\end{aligned}
\end{equation}
by testing \eqref{weakFokkerPlanckEq} with $\bT$ and using \eqref{apriori1}. 

Now, for a strong solution of the solvent-structure subsystem \eqref{divfree} and \eqref{momEq}-\eqref{shellEQ}, it has been shown in  \cite[Theorem  4.1]{BMSS})  that if
 \begin{align*}
\eta \in L^\infty(I:C^1(\omega)), \quad \bu\in L^{r }(I;L^{s }(\Oeta)), \qquad \frac{2}{r }+\frac{3}{s }\leq 1,
 \end{align*}
then the following global acceleration bound holds
\begin{equation}
\begin{aligned}
\label{acceleration}
\sup_{t\in I}&
\bigg( 
\Vert \nabx\bu(t)\Vert_{L^2(\Oeta)}^2 
+
 \Vert\partial_t \naby\eta(t)\Vert_{L^2(\omega)}^2 
+ 
\Vert\naby\dely\eta(t)\Vert_{L^2(\omega)}^2
\bigg)
\\&+
\int_I \big(
\Vert \nabx^2 \bu \Vert_{L^2(\Oeta)}^2
+
\Vert \partial_t \bu \Vert_{L^2(\Oeta)}^2
+
\Vert \nabx p \Vert_{L^2(\Oeta)}^2
+
\Vert\partial_t\dely \eta \Vert_{L^2(\omega)}^2
+
\Vert\partial_t^2 \eta \Vert_{L^2(\omega)}^2
\big)\dt
\\&\lesssim 
\Vert \nabx \bu_0\Vert_{L^2(\Omega_{\eta_{0}})}^2 
+
 \Vert \naby\eta_\star\Vert_{L^2(\omega)}^2 
+ 
\Vert\naby\dely\eta_0\Vert_{L^2(\omega)}^2
\\&
+
\int_I\Vert \bff\Vert_{L^2(\Oeta)}^2\dt
+
\int_I\Vert \nabx \bT\Vert_{L^2(\Oeta)}^2\dt
+
\int_I\Vert g\Vert_{L^2(\omega)}^2\dt.
\end{aligned}
\end{equation}
The constant in the bound depends only on the right-hand side of \eqref{apriori1} and we also note that by \eqref{strgEstFP3xx}  
\begin{equation}
\begin{aligned}
\label{tEstAl}
\int_I\Vert \nabx \bT\Vert_{L^2(\Oeta)}^2\dt
\lesssim
\Vert\bT_0\Vert_{L^2(\Omega_{\eta_0})}^2
+
T \Vert\rho_0  \Vert_{L^2(\Omega_{\eta_{0}})}^2 .
\end{aligned}
\end{equation} 
We now test \eqref{rhoEqu}  with $\Delx \rho$. Just as in \eqref{spaceRegRho1}, this yields
\begin{equation}
\begin{aligned}
\label{spaceRegRho1xx}
\sup_{t\in I}\Vert \nabx\rho(t)\Vert_{L^2(\Oeta)}&+\int_I\Vert \Delx\rho\Vert_{L^2(\Oeta)}\dt
\\&\lesssim
\Vert \nabx\rho_0\Vert_{L^2(\Omega_{\eta_0})}
+
\int_I\big(\Vert\partial_t\eta  \Vert_{W^{2,2}(\omega)}^2+\Vert\bu\Vert_{W^{2,2}(\Oeta)}^2\big)\dt
\end{aligned}
\end{equation}
and by testing \eqref{solute} with $\Delx \bT$, similar to \eqref{spaceRegT1}, we obtain 
\begin{equation}
\begin{aligned}
\label{spaceRegT1xx}
\sup_{t\in I}&\Vert \nabx\bT(t)\Vert_{L^2(\Oeta)}
+
\int_I\Vert \Delx\bT\Vert_{L^2(\Oeta)}\dt
\\&\lesssim
\Vert \nabx\bT_0\Vert_{L^2(\Omega_{\eta_0})}
+
\int_I\big(\Vert\partial_t\eta  \Vert_{W^{2,2}(\omega)}^2+\Vert\bu\Vert_{W^{2,2}(\Oeta)}^2\big)\dt
\\&
+ \Vert\bT_0\Vert_{L^2(\Omega_{\eta_0})}^2
+
T \Vert\rho_0  \Vert_{L^2(\Omega_{\eta_{0}})}^2 
.
\end{aligned}
\end{equation}
Finally, by testing \eqref{rhoEqu}  with $\partial_t \rho$ and testing \eqref{solute} with $\partial_t\bT$, similar to \eqref{timeRegRho1} and \eqref{timeRegT1}, we obtain
\begin{equation}
\begin{aligned}
\label{timeRegRho1xx}
\int_I\Vert \partial_t \rho\Vert_{L^2(\Oeta)}^2\dt
&+
\sup_{t\in I} \Vert \nabx \rho(t)\Vert_{L^2(\Oeta)}^2
\\&\lesssim
\Vert \nabx\rho_0\Vert_{L^2(\Omega_{\eta_0})}
+
\int_I\big(\Vert\partial_t\eta  \Vert_{W^{2,2}(\omega)}^2+\Vert\bu\Vert_{W^{2,2}(\Oeta)}^2\big)\dt.
\end{aligned}
\end{equation}
and
\begin{equation}
\begin{aligned}
\label{timeRegT1xx}
\int_I\Vert \partial_t \bT\Vert_{L^2(\Oeta)}^2\dt
&+
\sup_{t\in I} \Vert \nabx \bT(t)\Vert_{L^2(\Oeta)}^2
\\&\lesssim
\Vert \nabx\bT_0\Vert_{L^2(\Omega_{\eta_0})}
+
\int_I\big(\Vert\partial_t\eta  \Vert_{W^{2,2}(\omega)}^2+\Vert\bu\Vert_{W^{2,2}(\Oeta)}^2\big)\dt
\\&
+
 \Vert\bT_0\Vert_{L^2(\Omega_{\eta_0})}^2
+
T \Vert\rho_0  \Vert_{L^2(\Omega_{\eta_{0}})}^2 
.
\end{aligned}
\end{equation}
respectively.
If we now combine \eqref{spaceRegRho1xx}, \eqref{spaceRegT1xx} \eqref{timeRegRho1xx} and \eqref{timeRegT1xx}, we obtain
\begin{equation}
\begin{aligned} 
\int_I\big(\Vert \partial_t \rho\Vert_{L^2(\Oeta)}^2
&+
\Vert \partial_t \bT\Vert_{L^2(\Oeta)}^2
\big)\dt
+
\sup_{t\in I} \big(\Vert \rho(t)\Vert_{W^{1,2}(\Oeta)}^2
+
\Vert \bT(t)\Vert_{W^{1,2}(\Oeta)}^2
\big)
\\&
+\int_I\big(\Vert  \rho\Vert_{W^{2,2}(\Oeta)}
+
\Vert  \bT\Vert_{W^{2,2}(\Oeta)}\big)
\\&\lesssim
\Vert  \rho_0\Vert_{W^{1,2}(\Omega_{\eta_0})}
+
\Vert  \bT_0\Vert_{W^{1,2}(\Omega_{\eta_0})}
+
\int_I\big(\Vert\partial_t\eta  \Vert_{W^{2,2}(\omega)}^2+\Vert\bu\Vert_{W^{2,2}(\Oeta)}^2\big)\dt.
\label{eq:thm:mainFPxxx}
\end{aligned}
\end{equation}   
Note that we also obtain from \eqref{shellAlone}, trace theorem, \eqref{acceleration} and \eqref{tEstAl} 
\begin{equation}
\begin{aligned}
\label{higherShellReg}
\int_I\Vert\dely^2 \eta\Vert_{L^2(\omega)}^2\dt
\lesssim&
\int_I\Vert g\Vert_{L^2(\omega)}^2\dt
+ 
\int_I\Vert\partial_t^2 \eta\Vert_{L^2(\omega)}^2\dt 
+
\int_I\Vert \partial_t\dely \eta \Vert_{L^2(\omega)}^2\dt
\\&
+
\int_I \big(
\Vert \nabx^2 \bu \Vert_{L^2(\Oeta)}^2 
+
\Vert \nabx p \Vert_{L^2(\Oeta)}^2
+
\Vert \nabx \bT \Vert_{L^2(\Oeta)}^2
\big)\dt
\\
\lesssim& 
\Vert \nabx \bu_0\Vert_{L^2(\Omega_{\eta_{0}})}^2 
+
 \Vert \naby\eta_\star\Vert_{L^2(\omega)}^2 
+ 
\Vert\naby\dely\eta_0\Vert_{L^2(\omega)}^2
\\&
+
\int_I\Vert \bff\Vert_{L^2(\Oeta)}^2\dt
+
\int_I\Vert g\Vert_{L^2(\omega)}^2\dt
+
\Vert\bT_0\Vert_{L^2(\Omega_{\eta_0})}^2
+
T \Vert\rho_0  \Vert_{L^2(\Omega_{\eta_{0}})}^2.
\end{aligned}
\end{equation} 
Finally, combining \eqref{acceleration}, \eqref{eq:thm:mainFPxxx} and \eqref{higherShellReg} yields the desired estimate.

 \end{proof}

\section*{Statements and Declarations} 
%\subsection*{Acknowledgements}
%The author would like to thank Dominic Breit for many useful discussions.
%\subsection*{Funding}
%This work has been supported by the European Research Council (ERC) Synergy grant STUOD-DLV-856408.
\subsection*{Author Contribution}
The author wrote and reviewed the manuscript.
\subsection*{Conflict of Interest}
The author declares that they have no conflict of interest.
\subsection*{Data Availability Statement}
Data sharing is not applicable to this article as no datasets were generated
or analyzed during the current study.
\subsection*{Competing Interests}
The author have no competing interests to declare that are relevant to the content of this article.
%****List of advanced bibliographystyle****
% 1. spbasic
% 2. spphys
% 3. spmpsci
%%
%
%\bibliographystyle{spmpsci}
%\bibliography{myBibliography}
%
%

\end{document}